\documentclass[10pt]{amsart}
\usepackage{graphicx,amssymb,amsfonts,amsmath,amsthm,newlfont}
\usepackage{epsfig,url}
\usepackage{color}

\usepackage[all,2cell]{xy} \UseAllTwocells \SilentMatrices

\newtheorem{thm}{Theorem}[section]
\newtheorem{cor}[thm]{Corollary}
\newtheorem{lem}[thm]{Lemma}
\newtheorem{prop}[thm]{Proposition}
\theoremstyle{definition}
\newtheorem{defn}[thm]{Definition}
\newtheorem{conj}{Conjecture}

\newtheorem{example}[thm]{Example}

\theoremstyle{remark}
\newtheorem{rem}[thm]{Remark}
\numberwithin{equation}{section}

\newcommand{\Z}{\mathbb Z}
\newcommand{\C}{\mathbb C}

\newcommand{\R}{\mathbb R}
\newcommand{\N}{\mathbb N}
\newcommand{\Pro}{\mathbb P}

\newcommand{\bbf}{\overline{b}}

\newcommand{\gr}{\mathrm{gr}}

\newcommand{\MT}{\mathcal{MT}}
\newcommand{\M}{\mathcal{M}}

\newcommand{\Der}{\mathrm{Der}}
\newcommand{\cusp}{\mathrm{cusp}}

\newcommand{\zetam}{\zeta^{ \mathfrak{m}}}

\newcommand{\stu}{*}

\newcommand{\ad}{\mathrm{ad}}

\newcommand{\studot}{\underline{\cdot}\,}

\newcommand{\cc}{\mathfrak{c}}

\newcommand{\eact}{\circledast}

\newcommand{\Q}{\mathbb Q}

\newcommand{\To}{\longrightarrow}

\newcommand{\G}{\mathbb{G}}

\newcommand{\x}{\mathsf{x}}

\newcommand{\tone}{\overset{\rightarrow}{1}\!}
\newcommand{\opi}{{}_0 \Pi_{1}}
\newcommand{\ipi}{{}_1 \Pi_{1}}

\newcommand{\ue}{\mathfrak{u}^{\mathrm{geom}}}

\newcommand{\Or}{\mathcal{O}}

\newcommand{\g}{\mathfrak{g}^{\mathfrak{m}}}

\newcommand{\circb}{\, \underline{\circ}\, }

\newcommand{\Res}{\mathrm{Res}}

\newcommand{\SL}{\mathrm{SL}}

\newcommand{\ls}{\mathfrak{ls}}

\newcommand{\mm}{\mathfrak{m} }

\newcommand{\HH}{\mathbb{H} }

\newcommand{\e}{\mathbf{e}}

\newcommand{\Lie}{\mathrm{Lie}\,}

\newcommand{\LL}{\mathbb{L}}

\newcommand{\dmr}{\mathfrak{dmr}}

\newcommand{\per}{\mathrm{per}}

\newcommand{\PP}{\mathcal{P}}

\newcommand{\dg}{\mathfrak{d}}

\newcommand{\hsigma}{\underline{\sigma}}

\newcommand{\rel}{\mathrm{rel}}

\newcommand{\eis}{\mathrm{eis}}

\newcommand*{\longhookrightarrow}{\ensuremath{\lhook\joinrel\relbar\joinrel\rightarrow}}

\begin{document}
\author{Francis Brown}
\begin{title}[Zeta elements   and  the   Lie algebra  of $\pi^1(E \backslash 0)$]{Zeta elements in depth 3 and  the   fundamental Lie algebra of a punctured  elliptic curve}\end{title}
\maketitle

\begin{abstract}  
This paper draws   connections between the double shuffle equations and structure of associators;
 universal mixed elliptic motives as defined by Hain and Matsumoto; and the Rankin-Selberg method for modular forms for $\SL_2(\Z)$.  We write down explicit formulae for zeta elements $\sigma_{2n-1}$ (generators of the Tannaka Lie algebra of the category of mixed Tate motives over $\Z$)  in depths up to four, give  applications to the Broadhurst-Kreimer conjecture, and  completely  solve 
 the double shuffle equations for multiple zeta values in depths two and three.
\end{abstract}

\section{Introduction}

  The  theme  of this paper is that certain constructions relating to the motivic fundamental group of the projective line minus 3 points, which are inherently ambiguous,  
  can be  explicitly determined, and simplified, by passing to  genus one.

 The main result  can be viewed on the following three different levels. 
 
 \subsection{The fundamental group of $\Pro^1 \backslash \{0,1,\infty\}$} The de Rham   fundamental group
 $$\ipi=\pi_1^{dR}( \Pro^1\backslash \{0,1,\infty\}, \tone_1)$$
 of  $\Pro^1\backslash \{0,1,\infty\}$ with tangential base point the unit tangent vector at $1$
 is a prounipotent affine group scheme over $\Q$. Its graded Lie algebra  is the free Lie algebra $\LL(\x_0,\x_1)$
 on two generators $\x_0,\x_1$ corresponding to loops around $0$ and $1$.   Since $\ipi$ is a (pro) object in the category of mixed Tate motives over $\Z$, it admits  an action of the Tannakian fundamental group $G^{dR}_{\MT(\Z)}$. Denote the graded Lie algebra of the latter by 
  \begin{equation} \label{introgdef}
  \g = \LL\langle \sigma_3, \sigma_5, \ldots \rangle\ .
  \end{equation}
  It is the free graded Lie  algebra generated by non-canonical elements
 $\sigma_{2n+1}$  in degree $-2n-1$ for $n\geq 1$.
  We obtain 
a morphism of Lie algebras
 \begin{equation} \label{introgtoDer1}
  i_0:  \g \To \Der^1\,  \LL(\x_0,\x_1) 
 \end{equation}
  where 
   $\Der^1\, \LL(\x_0,\x_1)$ denotes the set of derivations which send $\x_1$ to $0$. Furthermore, we know that 
 $(\ref{introgtoDer1})$ factors through a morphism
 $$i : \g \To \LL(\x_0,\x_1)$$
 and $(\ref{introgtoDer1})$ maps  $\sigma \in \g$  to the derivation which sends  $\x_0 \mapsto [i(\sigma), \x_0]$ and $\x_1$ to $0$.
 The main result of \cite{BrMTZ}  states that $i$ is injective, and therefore enables us to expand  elements $\sigma\in \g$ in `coordinates' $\x_0$ and $\x_1$. It is an important problem to characterise the image of the map $i$ and to describe the $i(\sigma_{2n-1})$ as explicitly as possible.  It is known by the work of Racinet that the image of $i$ is contained in the Lie algebra $\dmr_0$ of solutions to the double shuffle equations. It is 
 also contained in the space of solutions to Drinfeld's associator equations, which by a result of Furusho \cite{Fu},  are contained in $\dmr_0$.

  It is well-known that
    \begin{equation} \label{introCannorm} 
 i(\sigma_{2n+1}) =  \ad(\x_0)^{2n} \x_1 + \hbox{ terms of degree } \geq 2 \hbox{ in } \x_1\  ,
 \end{equation} 
This can be seen as follows.
   The Drinfeld associator is
    the formal power  series which is the generating series of shuffle-regularised multiple zeta values, 
 $$Z = \sum_{ w\in \{\x_0,\x_1\}^{\times}} \zeta(w) w  \qquad \in \quad \R \langle \langle \x_0, \x_1\rangle \rangle \ ;$$
and is easily computed  explicitly to first order in $\x_1$:
 $$ Z \equiv    \sum_{n\geq 2}\,  \zeta(n) \, \ad(\x_0)^{n-1} \x_1  \pmod{ \hbox{ terms of degree } \geq 2 \hbox{ in } \x_1}\ . $$
The coefficients of the odd zeta values in degree $2n+1$ are congruent to $i(\sigma_{2n+1})$, 
 but very little is known in general  about the coefficients of $i(\sigma_{2n+1})$ of  higher degrees $\geq 3$ in the $\x_1$. 
 In this paper, we show:
\begin{thm}
 \begin{enumerate}
 \item That there is  a  choice of generators $\sigma^c_{2n+1}\in \g$  which is given by an  explicit formula $(\ref{introexplicitforsigma})$  modulo terms of degree $\geq 5$ in $\x_1$.
 \item  That there is  a  rational associator   $\tau$ which is given by an explicit formula modulo terms of degree $\geq 4$ in $\x_1$. It computes the coefficients of the even zeta values (powers of $\pi$) in $Z$ in this range.
  \end{enumerate} 
 \end{thm}
 Statement $(1)$ is surprising because the choice of generators $\sigma_{2n+1}$ are \emph{a priori} only well-defined up to addition 
 of higher order commutators of $\sigma_{2m+1}$.  The key point is that by passing to genus $1$, we find canonical coordinates on 
 $\g$ which enable us to  fix  the  commutators $[\sigma_{2a+1}, [\sigma_{2b+1}, \sigma_{2c+1}]]$ uniquely.  A similar story holds for $(2)$.

  In so doing, we discover  that it is more convenient to consider a different normalisation for the $\sigma_{2n+1}$ from the canonical normalisation $(\ref{introCannorm})$, which we call   the \emph{heretical normalisation}
  \begin{equation} \label{introHereticalnorm} 
 \hsigma_{2n+1} = {B_{2n} \over (2n)!}\, \ad(\x_0)^{2n} \x_1 + \hbox{ terms of degree } \geq 2 \hbox{ in } \x_1\  ,
 \end{equation} 
 where $B_{2n}$ is the $2n^{\mathrm{th}}$ Bernoulli number.   Objects which are normalised according to the heretical normalisation 
 will be underscored.

\subsection{The fundamental group of the first-order Tate curve}
Let $E^{\times}_{\partial/\partial q}$ denote the fiber of the universal elliptic curve $\mathcal{M}_{1,2} \rightarrow \mathcal{M}_{1,1}$ over the tangential base point
${\partial \over \partial q}$ on $\mathcal{M}_{1,1}$, where $\mathcal{M}_{g,n}$ denotes the moduli space of curves of genus $g$ with $n$ marked points.  
In a future paper with Hain, we shall show (as suggested in \cite{MEM}) that its de Rham fundamental group 
\begin{equation} \label{introPPdefn} 
\PP =  \pi_1^{dR} ( E_{\partial/\partial q}^{\times}, \tone_1)
\end{equation}
where $\tone_1$ is the tangent vector of length $1$ with respect to the canonical holomorphic coordinate $w$ on $E_{\partial/\partial q}^{\times}$,
is the de Rham realisation of a pro-object in the category of mixed Tate motives over $\Z$. 
Its associated bigraded  Lie algebra (bigraded for the weight  $W$ and relative monodromy-weight $M$  filtrations) is the free Lie algebra
on certain  canonical generators $a,b$. 
Correspondingly\footnote{If one thinks  of $\g$ as being bigraded for $M$ and $W$, with $W=M$,   then the map $i_1$ respects the $M$-grading, but not the $W$-grading, see \cite{MEM}.}, one obtains a morphism of Lie algebras
\begin{equation}  \label{introi1def}
i_1: \g \To \Der^{\Theta} \, \LL(a,b) 
\end{equation} 
 where $\Der^{\Theta}$ denotes the set of derivations $\delta$ such that $\delta(\Theta)=0$, where $\Theta = [a,b]$. 
 We shall show as a consequence of \cite{BrMTZ} that $(\ref{introi1def})$ is injective.

 On the other hand, there are distinguished   derivations $\varepsilon_{2n}^{\vee}\in \Der^{\Theta} \, \LL(a,b)$  whose action on $a$ are given by 
$$\varepsilon_{2n}^{\vee} a   =    \ad(a)^{2n} b \nonumber \qquad \hbox{ for }  n\geq 1\ . $$ 
They were first studied by Nakamura \cite{Na} in  a slightly different  context and rediscovered in \cite{CEE, LR}.  The  action of $\varepsilon_{2n}^{\vee}$ on $b$ is determined by the 
condition $\varepsilon_{2n}^{\vee} \Theta=0$ together with the fact that it is homogeneous of degree $2n$ in $a,b$.
The derivations $\varepsilon_{2n}^{\vee}$ are `geometric' in the sense that 
the relative  completion of 
$\SL_2(\Z)= \pi_1(\mathcal{M}_{1,1}, {\partial / \partial q})$ (or universal monodromy) 
 acts on $\LL(a,b)$ via the Lie algebra generated by the $\varepsilon_{2n}^{\vee}$ and their images $\ad(\varepsilon_0^{\vee})^k \varepsilon^{\vee}_{2n}$ under the adjoint action of
 $$\varepsilon_0^{\vee} \in \Der^{\Theta} \,  \LL(a,b) \quad ,  \qquad \varepsilon_0^{\vee}(a)=b  \quad \varepsilon_0^{\vee}(b)=0 \ .$$
 Denote the Lie subalgebra  generated by the $\varepsilon_{2n}^{\vee}$, for all  $n\geq 0$, by   
 $$\ue  \subset \Der^{\Theta} \LL(a,b)\ .$$ 
 It is the bigraded  image of the universal monodromy \cite{HaGPS}.
 The elements $\varepsilon_{2n}^{\vee}$ satisfy many relations which were studied by Pollack \cite{Po}.
 The image of $\g$ in $\Der^{\Theta} \, \LL(a,b)$ under $(\ref{introi1def})$ is  by no means contained in $\ue$, but in low degrees with respect to $b$, the $\varepsilon_{2n}^{\vee}$ give canonical `coordinates' in which to write down the initial terms of elements $i_1(\sigma_{2n+1})$.
The motivic version of a formula due to Nakamura in the $\ell$-adic setting, is 
$$i_1( \sigma_{2n+1}) \equiv \varepsilon_{2n+2}^{\vee} \pmod {W_{-2n-3}}$$
for all $n \geq 1$. A more precise result can be obtained using the elements $\sigma^c_{2n+1}$. 
\begin{thm} \label{introthmsigmaasepsilons} Let $n\geq 2$. The elements $\sigma^c_{2n+1}$ 
satisfy
\begin{equation} \label{introexplicitsigmaasepsilons}
i_1(\underline{\sigma}_{2n+1}^c) \equiv \underline{\varepsilon}^{\vee}_{2n+2} +  
\sum_{a+b=n}  {1 \over 2 b }  [\underline{\varepsilon}^{\vee}_{2a+2}, [\underline{\varepsilon}^{\vee}_{2b+2}, \underline{\varepsilon}^{\vee}_{0} ]] \quad  \pmod{W_{-2n-5}}\ .
\end{equation}
where the $\underline{\varepsilon}^{\vee}_{2n}$ are the heretical normalisations  $(\ref{hereticalepsilon})$ of the $\varepsilon_{2n}^{\vee}$.
\end{thm}

This theorem  is equivalent to an explicit formula for the $i_0(\sigma^c_{2n+1})  \in \LL(\x_0,\x_1)$ 
modulo terms of degree $\geq 5$ in $\x_1$, which is how it is proved here.
 This uses an explicit  morphism  from the de Rham  
fundamental group of  $\Pro^1\backslash \{0,1,\infty\}$ to  that of $E^{\times}$ which was written down by Hain. 
Since $i_0$ and $i_1$ are compatible with this morphism, the expansions of $\sigma^c_{2n+1}$ under $i_0$ and $i_1$
can be related to each other explicitly.

   \subsection{Rankin-Selberg method} The third way of understanding
   the elements $\sigma_{2n+1}^c$, and 
the starting point for this paper, came from the theory of iterated integrals of holomorphic modular forms  for $\SL_2(\Z)$. 
The   coefficients in equation $(\ref{introexplicitsigmaasepsilons})$ come from the computation  \cite{MMV} of the imaginary part of an iterated integral of two Eisenstein series using the Rankin-Selberg method. They turn out to be  the coefficients of $\zeta(2n-1)$ in 
 the convolution  of two Eisenstein series of different weights, which are products of Bernoulli numbers. Equivalently, they are proportional to  the coefficients in the odd period polynomials of  Eisenstein series. This is the motivation for the heretical normalisations $(\ref{introHereticalnorm})$ and the source of the 
 formula   $(\ref{introexplicitsigmaasepsilons})$.

   \subsection{Further remarks} We discuss  the methods used in this paper, and applications  to the double 
   shuffle equations and Broadhurst-Kreimer conjecture.
   
   \subsubsection{Commutative power series and anatomy of associators}
One tool which we use extensively  is the method of commutative power series and  is  closely related to Ecalle's theory of moulds \cite{Ecalle1, Ecalle2}. Let $\LL(u,v)$ be the free graded
Lie algebra generated by two elements $u, v$.  It is graded for the degree in $v$.
Elements of $v$-degree $r\geq 1$  in the tensor algebra $T(u,v)$ can be encoded by commutative polynomials
\begin{eqnarray}
\rho:    \gr^r_v \, T(u,v) &  {\longhookrightarrow}  &\Q [ x_1,\ldots, x_r]  \qquad r\geq 1 \\ 
u^{i_0} v u^{i_1} \ldots v u^{i_r} & \mapsto & x_1^{i_1} \ldots x_r^{i_r} \nonumber
\end{eqnarray} 
We apply this construction to $(u,v) = (\x_0,\x_1)$ and $(u,v)= (a,b)$, and their derivation algebras. 
In certain contexts, we shall explain  that it is natural to rescale the morphism $\rho$ by introducing polynomial denominators. In this manner, elements 
of $\Der^1 \, \LL(\x_0,\x_1)$ and $\Der^{\Theta} \, \LL(a,b)$ are encoded by sequences of \emph{rational functions} in $x_1,\ldots, x_r$.  
The double shuffle equations (defining equations for the Lie algebra $\dmr_0$) can be translated into functional 
equations for commutative power series via the map $\rho$.  A surprising discovery is that there exist  canonical solutions if one allows poles:

\begin{thm} \cite{AA} There exist explicit solutions to the double shuffle equations in the space of rational functions
in all weights and all depths.
\end{thm} 

There is a particular family of solutions we denote by  $\xi^{(r)}_{2n+1} \in \Q(x_1,\ldots, x_r)$ in weight $2n+1\geq 3$. Their components in  depths
$r=1,2$  are polynomials, but they have
poles in depths  $r\geq 3$. Furthermore, a new element emerges in weight $-1$  which we denote by 
 $\xi^{(r)}_{-1} \in \Q(x_1,\ldots, x_r) $.  The idea of \cite{AA}  is to expand the polynomial representation of  zeta ements $\rho(i(\sigma_{2n-1}))$ in terms of the $\xi_{2n+1}$.
This `anatomy'  can be computed explicitly in low degrees:
 
 \begin{thm} \label{introthmanatomy} If $\{ \  , \  \}$ denotes the Ihara bracket, extended to the setting of rational functions, then the canonical zeta elements up
 to depth 4 are given, in the heretical normalisation,  by the simple formula:
\begin{equation} \label{introexplicitforsigma} 
\rho(i( \underline{\sigma}^c_{2n+1} ))\equiv \underline{\xi}_{2n+1} + \sum_{a+b=n}  {1 \over 2 b } \{ \underline{\xi}_{2a+1}, \{\underline{\xi}_{2b+1}, \underline{\xi}_{-1} \}\} \quad  \pmod{\hbox{depths } \geq 5}\ .
 \end{equation}
 \end{thm}
 If one were to write this formula in the canonical, as opposed to heretical normalisations, one would find  coefficients given by products of Bernoulli numbers  in the sum in the right-hand side. These coefficients are essentially the coefficients in the odd period polynomial of Eisenstein series, a fact which emerges from \S\ref{sectfinal}. 

Theorem $\ref{introexplicitforsigma}$ is proved by combinatorial methods, and uses Goncharov's theorem that the solutions to the double
shuffle equations in depth 3 are motivic. It makes no reference to the first-order Tate curve.
It is more illuminating, however, to interpret this theorem by passing to genus $1$. Via the Hain morphism $\ref{sectHainmorphism}$, it turns out that the elements
$\xi_{2n+1}$  correspond in low depths  to the derivations $\varepsilon_{2n+2}^{\vee}$, and this proves that
theorems  $\ref{introthmanatomy}$ and  $\ref{introthmsigmaasepsilons}$ are equivalent.

      \subsubsection{Double shuffle equations}
Our elements $\sigma_{2n+1}^c$ are explicit solutions to the double shuffle equations in depths $\leq 4$ and odd weights.
We also construct, in \S\ref{secttau}, an explicit solution $\tau$ in depths $\leq 3$ and all even weights.  Using a theorem due to Goncharov
computing the dimension of the space of solutions to linearised double shuffle equations in depth 3 we deduce the

 \begin{thm}  Every solution to the regularised double shuffle equations in depths $\leq 4$ (odd weight) and depths $\leq 3$ (even weight) 
 can be expressed using the explicit elements $\sigma_{2n+1}^c$ and the rational associator $\tau$.
 \end{thm} 
 
This theorem can be applied to the semi-numerical algorithm  described in \cite{BrDec}  
 to decompose motivic multiple zeta values into a  chosen basis
  using the motivic coaction. It involved a numerical computation of a regulator  at each step. 
 One application of the elements $\sigma_{2n+1}^c$ and $\tau$ is to remove this transcendental
 step, leading  to an exact and  effective algorithm for proving motivic relations between   multiple zeta values in low depths, and any  weight.
  It replaces   the need to store tables of multiple zeta values  in this range of depths \cite{DataMine}.

A further manifestation of the double shuffle equations occurs in genus 1.
As mentioned above, we encode  elements of $\Der^{\Theta} \, \LL(a,b)$ by rational functions, by   composing the morphism
$ \delta \mapsto \delta(a): \Der^{\Theta}\,  \LL(a,b) \To \LL(a,b) \subset T(a,b)$
with the linear map
\begin{eqnarray}\label{introgrbtoratfunc}
\gr^{r}_b\, T(a,b) &  \To &   \Q(x_1,\ldots, x_r)    \\
a^{i_0} b a^{i_1} b \ldots b a^{i_r} & \mapsto &    { x_1^{i_1} \ldots x_r^{i_r}  \over x_1(x_1-x_2) \ldots (x_{r-1}-x_{r})x_r  } \ . \nonumber
\end{eqnarray} 
In \cite{AA}, we defined a  bigraded Lie algebra $\mathfrak{pls}$ to be the space of solutions to the linearised double shuffle equations with poles  at worst of the above form. We show that:

\begin{prop} The Lie algebra of geometric derivations is contained, via $(\ref{introgrbtoratfunc})$, in the space solutions to the linearised double shuffle equations:
$\ue \subset \mathfrak{pls}$ \end{prop} 

Thus the linearised double shuffle equations occur naturally in the elliptic setting. 
It is natural to ask   if $\ue= \mathfrak{pls}$,  and easy to show \cite{AA} that this holds in depths $\leq 3$.
It follows from the previous proposition that  the stuffle equations enable us to detect non-geometric derivations, i.e., elements in  the quotient
$$ (\Der^{\Theta} \,\LL(a,b)) / \ue\ . $$

\subsubsection{Depth 4 generators in the Broadhurst-Kreimer conjecture}
A further application of the canonical elements $\sigma^c_{2n+1}$ is to the Broadhurst-Kreimer conjecture.

It is well-known since Ihara and Takao   \cite{IT} that there exist quadratic relations 
\begin{equation}\label{introquadrel} 
\sum_{i, j} \lambda_{i, j} \{ \sigma_{2i+1} , \sigma_{2j+1} \} \equiv 0 \pmod{ \hbox{terms of degree } \geq 4 \hbox{ in } \x_1} 
\end{equation} 
where $\lambda_{i,j}\in \Q$ are coefficients of period polynomials  $P$ of even, cuspidal $\SL_2(\Z)$-cocycles. 
In \cite{BrDepth}, we reformulated the Broadhurst-Kreimer conjecture, which describes the dimensions of the space of multiple zeta values
graded by the depth, in terms of the spectral sequence induced on $\g$ by the depth filtration $D$. Using the elements $\sigma_{2n+1}^c$
we can compute the first non-trivial differential (conjecturally, the only non-trivial differential) in this spectral sequence. 
A motivic version of the   Broadhurst-Kreimer conjecture can thus be  formulated by saying that  $\gr_{D} \g$ has the following presentation: it is generated by 
the classes $[\sigma_{2n+1}] \in  \gr^1_{D} \g$ for all $n\geq 1$, and  certain elements $\cc(P) \in  \gr^4_{D} \g$, where $\cc$ can be expressed in terms of the canonical 
elements $\sigma^c_{2n+1}$, and it is  subject only  to the  relations $(\ref{introquadrel})$. For a precise statement, see \S\ref{sectCuspelements}. 
This formulation of the Broadhurst-Kreimer conjecture can  be transposed to the elliptic setting  and describes a presentation for a
certain Lie subalgebra of $\ue$.
  \\

\emph{Acknowledgements}.
 This paper was written during a stay at the IAS, partial supported by NSF grant  DMS-1128155
  and ERC grant 257638.  Its origin was  an attempt to interpret the algebraic structures in the notes \cite{AA} geometrically.
    This paper owes a great deal to Richard Hain, who kindly explained his recent papers \cite{HaGPS}, \cite{MEM} to me
  during inumerable discussions, and  made many important  suggestions. Several of the topics touched upon here will be expanded in a future joint work.
Many thanks also to Ding Ma for corrections.

\newpage
\section{Reminders on the projective line minus 3 points}
Background material can be found in \cite{DG}, \cite{Ra}, \cite{BrDepth}.
\subsection{Depth} Let $\LL(\x_0,\x_1)$ denote the free  graded Lie algebra over $\Q$ on two generators $\x_0,\x_1$, where $\x_0$ and $\x_1$ have degree  $1$.
The depth filtration $D^n \LL(\x_0,\x_1)$   is the decreasing filtration such that $D^0= \LL(\x_0,\x_1)$ and 
$$D^1 \LL(\x_0,\x_1) = \ker  ( \LL(\x_0,\x_1) \To \LL(\x_0))$$
where the map on the right sends $\x_1$ to $0$ and $\x_0$ to $\x_0$. It is defined by $D^n = [D^1, D^{n-1}]$ for all $n\geq 2$. It is the decreasing filtration associated to the $D$-degree, for which 
$\x_0$ has $D$-degree $0$ and $\x_1$ has $D$-degree $1$.
Therefore $D^n \LL(\x_0,\x_1)$ consists
of $\Q$-linear combinations of Lie brackets of $\x_0$ and $\x_1$ with at least $n$ $\x_1$'s. 

The universal enveloping algebra of $\LL(\x_0,\x_1)$ is the graded tensor algebra  $T(\x_0,\x_1)$ on $\Q \x_0 \oplus \Q\x_1$. 
The $D$-degree is defined in the same manner on $T(\x_0,\x_1)$ and defines a decreasing filtration  $D^n T(\x_0,\x_1)$ spanned by  words in $\geq n$ $\x_1$'s. We shall embed 
$ \LL(\x_0,\x_1)  \subset T(\x_0,\x_1)$; the embedding is compatible with the filtrations $D$.

\subsection{Ihara bracket} \label{sectIhara}
The de Rham fundamental groupoid \cite{DG}
 $${}_0\Pi_1 = \pi_1^{dR}( \Pro^1 \backslash \{0,1,\infty\}, \tone_0, -\tone_1)$$
 is the de Rham realisation of a mixed Tate motive over $\Z$, and admits an action of the de Rham  motivic Galois group $G^{dR}$. The action on the trivial de Rham path ${}_0 1_1 $
 from the tangential base point $1$ at $0$ to the tangential base point $-1$ at $1$ gives a morphism
 \begin{equation} \label{GdRtoop1}
 g\mapsto g.{}_0 1_1:  G^{dR} \To {}_0\Pi_1 
 \end{equation}
 of schemes. It becomes a morphism of groups if one equips ${}_0 \Pi_1$ with the Ihara group law, which is denoted by $\circ$. If $R$ is a commutative unitary algebra, the set of $R$-points of ${}_0 \Pi_1$ is the set of invertible group-like (with respect to the completed coproduct  
 for which $\x_0,\x_1$ are primitive) formal power series $R\langle \langle \x_0, \x_1\rangle \rangle$
 in two non-commuting variables. The Ihara group law is then given by the formula
 \begin{eqnarray}
\circ:  \opi \times \opi &  \To & \opi  \\
 F \circ G & = &  G(\x_0, F \x_1 F^{-1} ) F \ .  \nonumber
  \end{eqnarray}
 The expression on the right-hand side is also equal to 
 $$F\circ G= F G( F^{-1} \x_0 F, \x_1)\ . $$    Likewise, the de Rham fundamental group  with tangential base point $-1$ at $1$
   $${}_1\Pi_1 = \pi_1^{dR}( \Pro^1 \backslash \{0,1,\infty\}, -\tone_1) $$
 admits an action of $G^{dR}$, which  can be shown to factorise  through the composition of the map $(\ref{GdRtoop1})$ with the  left-action
  \begin{eqnarray} \label{circ1}
\circ_1:  \opi \times \ipi &  \To & \ipi  \\
 F \circ_1 H & = &  H( F^{-1} \x_0 F, \x_1 )  \ .  \nonumber
  \end{eqnarray}
 Now pass to graded Lie algebras. The graded Lie algebras of both $\opi$ and $\ipi$ can be identified with 
 $\LL(\x_0,\x_1)$. Let $\g = \Lie\!^{\gr} \, U^{dR}$, where $U^{dR}$ is the unipotent radical of $G^{dR}= U^{dR} \rtimes \G_m$. Equation  $(\ref{GdRtoop1})$ gives a morphism
 \begin{equation} \label{gfraktoL}
 i:   \g \To(  \LL(\x_0,\x_1), \{, \})
 \end{equation} 
 where $\{\, , \, \}$ is the Ihara bracket, for which we give a formula below.
 Let $\Der^1 \, \LL(\x_0,\x_1)$ denote the Lie subalgebra of  derivations $\delta \in \Der \, \LL(\x_0,\x_1)$ which satisfy $ \delta(\x_1)=0$. 
 The left action $(\ref{circ1})$ is given on the level of graded Lie algebras by 
  \begin{eqnarray}
( \LL(\x_0,\x_1)  , \{ \, , \, \})   &  \To &  \Der^1 \, \LL(\x_0, \x_1)  \\
  \sigma  & \mapsto &    \begin{cases}  \x_0 \mapsto  [\x_0, \sigma]  \\  \x_1 \mapsto 0 \\ \end{cases}\ . \nonumber
\end{eqnarray} 

\begin{thm} \label{thmDIconj} \cite{BrMTZ} The morphism $(\ref{gfraktoL})$ is injective.
\end{thm} 
Because of this theorem, we can identify $\g$ with its image in $\LL(\x_0,\x_1)$ via $i$. 
The graded Lie algebra $\g$ is freely generated by elements $\sigma_{2n+1}$ in degree $-2n-1$, for all $n\geq 1$. Their images  in 
$\LL(\x_0,\x_1)$ are the zeta ements
\begin{equation} \label{DrinfeldElements}
i(\sigma_{2n+1}) = \mathrm{ad} (\x_0)^{2n} \x_1 + \hbox{ terms of depth } \geq 2 \ . 
\end{equation} 
They are not \emph{a priori} canonical for $n\geq 5$. However, the Hoffman-Lyndon basis for motivic multiple zeta values allows one to define 
canonical choices  $\sigma^{HL}_{2n+1}$ \ for the $\sigma_{2n+1}$ \cite{BrICM}. Very little is known about these elements,
except for the  coefficients of $(\x_0\x_1)^a \x_0 (\x_0\x_1)^b$ in $\sigma_{2a+2b+1}$ (which are independent of the choice of $\sigma_{2n+1}$),
which follows from \cite{BrMTZ} and \cite{Za232}.

\subsection{Linearized Ihara action and depth} \label{sectLinIhara}
In \cite{BrDepth}, we considered the following linearised version of the Ihara action. For any word 
$a \in  \{\x_0,\x_1\}^{\times}$,  let 
$$(a_1\ldots a_n)^* = (-1)^n a_n \ldots a_1\ .$$
\begin{defn} Define a $\Q$-bilinear map 
$$ \circb  : T(\x_0,\x_1)  \otimes_{\Q} T(\x_0,\x_1) \rightarrow  T(\x_0,\x_1) \ $$
inductively as follows. For   any words $a, w$ in $\x_0,\x_1$, and for any integer  $n\geq 0$,   let 
\begin{equation}\label{circbdef}
 a \circb  (\x_0^{n} \x_1 w) =        \x_0^n  a \x_1  w  +   \x_0^n \x_1 a^* w +\x_0^n \x_1 (a \circb w)   \end{equation}
with the initial condition  $a \circb \x_0^n =  \x_0^n \, a $.
 \end{defn} 
The antisymmetrization of the  map $\circb$ restricts to the Ihara bracket on $\LL(\x_0,\x_1)$:
\begin{eqnarray}
\{f, g\} = f\circb g -g \circb f \qquad \hbox{ for all } f, g \in \LL(\x_0,\x_1)  \ .
\end{eqnarray}
 It follows from this formula that the Ihara bracket is homogeneous for the $D$-degree (the fact that 
it respects the depth filtration follows from the geometric interpretation of the depth filtration using the embedding 
$\Pro^1\backslash \{0,1,\infty\} \subset \Pro^1 \backslash \{0,\infty\}$ \cite{DG}). 
If we embed $\g \hookrightarrow \LL(\x_0,\x_1)$ via $(\ref{gfraktoL})$ we can  define the depth filtration $D \g$ on $\g$ to be the decreasing 
filtration induced by the depth filtration on $\LL(\x_0,\x_1)$.

For later use, remark that if    $A(a,b,c) = a\circb ( b\circb c) - (a\circb b) \circb c$, then  
\begin{equation} \label{Assoc}
A(a,b,c)= A(b,a,c) 
\end{equation} 
for any $a,b,c \in T(\x_0,\x_1)$, which follows from the definitions, and implies that the linearised Ihara bracket satisfies the Jacobi identity. 
\subsection{Double shuffle equations} 
 The double shuffle equations are a family of equations satisfied by  multiple zeta values which are well-adapted to the depth filtration.
 In his thesis \cite{Ra}, Racinet  defined a subspace
  $$\dmr_0 \subset \LL(\x_0,\x_1)\ ,$$ called the regularised double shuffle Lie algebra, which encodes these relations in terms
  of two Hopf algebra structures.  The Lie algebra version of Racinet's theorem is:

\begin{thm} \label{thmRacinet} \cite{Ra}  The space $\dmr_0$ is closed under the  Ihara bracket $\{ \, , \}$. \end{thm} 

Since the  regularised double shuffle equations hold for actual multiple zeta values, and are stable under the Ihara bracket, it follows that they are motivic. Combined with theorem \ref{thmDIconj} we deduce that there is an inclusion of Lie algebras
$$ \g \subset \dmr_0 \subset \LL(\x_0,\x_1)\ . $$
Therefore we can study elements $\sigma_{2n+1}$ by attempting to solve the defining equations of $\dmr_0$ in low depths.
This will  be achieved below using  the language of commutative power series (\S\ref{sectCommSeries}).

\subsection{Depth-graded motivic Lie algebra}
The depth filtration induces a decreasing filtration $D^{\bullet}$ on $G^{dR}$ and hence $\g$ via the maps $(\ref{GdRtoop1})$ and $(\ref{gfraktoL})$. Let 
\begin{equation} \dg = \gr^{\bullet}_D \g
\end{equation} denote the associated graded  Lie algebra. It is bigraded for weight and depth. The component of $\dg$
of depth $d$ and weight $n$ will be denoted by $\dg^d_n$. Let $\dg^d = \bigoplus_n \dg^d_{n}$, 
  denote the (infinite-dimensional) component in depth $d$,  and let $\dg_n = \bigoplus_d \dg^d_{n}$ denote the (finite-dimensional) component in weight $n$. 

The linearised double shuffle equations are a family of equations which were introduced in \cite{GKZ} and further studied in \cite{BrDepth}. It follows from a variant of Racinet's theorem that their set of solutions, denoted by $\ls\subset \gr^{\bullet}_D \LL(\x_0,\x_1)$, is a bigraded Lie algebra for the Ihara bracket $\{ \, , \, \}$,
and  a corollary of  theorem $\ref{thmDIconj}$ and Racinet's theorem  is 
\begin{thm} (\cite{BrDepth}, \S5)  $\dg \subset \ls.$
\end{thm} 
The following theorem was first proved by Tsumura. See \cite{BrDepth}, \S6.4 for a short proof. 
\begin{thm} \label{Depth-Paritythm} (Depth-Parity theorem for double shuffle equations).
We have 
$$\ls^d_n = 0 \quad \hbox{ if } \quad n \equiv d+1 \pmod{2}\ .$$
\end{thm} 
Combining the previous two theorems gives the
\begin{cor}  (Depth-Parity theorem). \label{corDepthParity} If $ n \equiv d+1 \pmod{2}$, then $\dg^d_n =0 $.
\end{cor}


\section{The fundamental Lie algebra of the  first-order Tate   curve} \label{sectElliptic}
The material in this section is an abridged account of results extracted from the papers \cite{HaGPS}, \cite{HaDe} and \cite{MEM}. 

\subsection{Background}  Let $E_{\partial/\partial q}^{\times}$ denote the  first order Tate curve, which is the fiber of the universal elliptic curve over $\M_{1,1}$ with respect to the tangential
base point $\partial/\partial q$. Its de Rham fundamental group 
\begin{equation} \label{PPdefn} 
\PP =  \pi_1^{dR} ( E_{\partial/\partial q}^{\times}, \tone_1)
\end{equation}
where $\tone_1$ is the tangent vector of length $1$ with respect to the canonical holomorphic coordinate $w$ on $E_{\partial/\partial q}^{\times}$,
is the de Rham realisation of a pro-object in the category of mixed Tate motives over $\Z$. Since its mixed Hodge structure is the limiting mixed Hodge structure of a variation, its relative monodromy weight filtration is denoted by $M$, and it comes equipped with  a geometric weight filtration $W$.  This data defines a universal mixed elliptic motive according to Hain and Matsumoto \cite{MEM}.  The associated $M, W$ bigraded Lie algebra is   the free Lie algebra $\LL(H_{dR})$ where 
$$H_{dR}=(H_{dR}^1( E_{\partial/\partial q}^{\times};\Q))^{\vee} = \Q a \oplus \Q b   \quad ( = \Q(1) \oplus \Q(0))$$
 which has   two canonical  de Rham generators $a$ and $b$.  It will be denoted by $\LL(a,b)$. The generators $a$ and $b$ have $(M,W)$ bidegrees $(-2, -1)$ and $(0,-1)$, respectively. The geometric weight filtration $W$ coincides with the lower central series filtration on $\PP$.

Since $\PP$ is the de Rham realisation of a mixed Tate motive over $\Z$, it admits an action of the de Rham motivic Galois group $G^{dR}$ of $\MT(\Z)$ 
 by the Tannakian formalism. Passing to Lie algebras gives  a morphism
\begin{equation}  \label{i1def}
i_1: \g \To \Der^{\Theta} \, \LL(a,b) 
\end{equation} 
where $\Theta = [a,b]$ and 
$$\Der^{\Theta} \,\LL(a,b) = \{ \delta \in \Der\, \LL(a,b): \delta(\Theta)= 0\}\ .$$
This is because $\Theta=[a,b]$ corresponds to the de Rham path which winds once around the puncture in $E_{\partial/\partial q}^{\times}$, and generates a copy of $\Q(1)$, which is fixed by $U^{dR}$.

\subsection{Derivations} 
We only need to consider  the subspace  $B^0 \, \Der^{\Theta} \,\LL(a,b) $ of derivations $\delta $  (see \S\ref{sectBdegree} for the definition of the $B$-filtration) such that 
$$\delta(b) \in  B^1 \LL(a,b):= \ker (\LL(a,b) \rightarrow  \LL(a))$$
where the map on the right sends $b$ to zero (it is the composition of the natural map $\LL(H_{dR}) \rightarrow \LL(H_{dR})^{ab}= H_{dR}$ followed by the projection $H_{dR} \rightarrow H_{dR}/F^0 H_{dR} =\Q a$). Such a derivation is uniquely determined by its value  $\delta(a)$
since $[\delta(a), b]+ [a,\delta(b)]=0$, and the commutator of $a$ is  $a\,\Q$. Thus $\delta \mapsto \delta(a)$ gives an embedding $B^0 \, \Der^{\Theta} \,\LL(a,b)  \rightarrow \LL(a,b)$. 

 For each $n\geq -1$, one shows that there exist elements 
$$\varepsilon_{2n+2}^{\vee} \in B^0   \, \Der^{\Theta} \,\LL(a,b)\subset \Der^{\Theta}\, \LL(a,b)$$
which are uniquely determined by the property
\begin{equation}
\varepsilon^{\vee}_{2n+2} (a)  =    \ad(a)^{2n+2} (b) \ .\nonumber 
 \end{equation} 
 The  elements $\varepsilon_{2n+2}^{\vee}$ were first defined by Nakamura in 1999 in the profinite setting.
 The element $\varepsilon^{\vee}_2$ is central in $ \Der^{\Theta}\, \LL(a,b)$ and plays no role here.    Let us define 
 \begin{equation}  \label{uedefn}
 \ue \subset \Der^{\Theta} \, \LL(a,b) 
 \end{equation} 
 to be the Lie subalgebra spanned by the $\varepsilon^{\vee}_{2n+2}$, for $n \geq -1$. 
 One shows  that the relative Mal\v{c}ev completion of  $\pi_1(\M_{1,1}, {\partial / \partial q})=\mathrm{SL}_2(\Z)$ (relative to $\SL_2(\Z) \rightarrow \SL_2(\Q)$) acts on $\LL(a,b)$ via $\ue$ \cite{HaGPS}.    Quadratic relations between the  elements $\varepsilon_{2n}^{\vee}$ predicted  by Hain and Matsumoto were studied by Pollack in his thesis \cite{Po}. 
  
 We define heretical normalisations  of these derivations as follows. Let
 \begin{equation} \label{hereticalepsilon}
\underline{\varepsilon}^{\vee}_{0} = {1\over 12 }  \varepsilon^{\vee}_{0} \qquad \hbox{ and } \qquad   \underline{\varepsilon}^{\vee}_{2n+2} = {B_{2n} \over (2n)!} \varepsilon^{\vee}_{2n+2} 
 \qquad \hbox{ for } n\geq 1 \end{equation} 
where  $B_{k}$ denotes the $k$th Bernoulli number.

\subsection{The Hain morphism} \label{sectHainmorphism} 
There is a morphism of fundamental groups \cite{HaDe}, \S16-18:
\begin{equation} \label{P1toEtimes} 
\pi_1(\Pro^1 \backslash \{0,1,\infty\}, \tone_1) \To \pi_1(E^{\times}_{\partial/\partial q}, {\partial / \partial w})\ .
\end{equation}
Using the work of Levin and Racinet, Hain has computed this map in the de Rham realisation  \cite{HaDe}, (18.1). On   de Rham   Lie algebras it is the continuous morphism
\begin{eqnarray} \label{phidefn}
\phi: \LL  (\x_0,\x_1 )^{\wedge} & \To &  \LL (a, b)^{\wedge} \\
\x_0 & \mapsto &  {\ad(b) \over e^{\ad(b)} -1}  a  =   a -{1 \over 2} [b,a] + {1 \over 12} [b,[b,a]] + \ldots\nonumber \\
\x_1 & \mapsto & [a,b] \nonumber
\end{eqnarray} 
where $\wedge$ denotes completion with respect to the lower central series.

 Let $ \Der^1 \LL(\x_0,\x_1)^{\wedge}$, $\Der^{\Theta} \, \LL(a,b)^{\wedge}$ denote the 
 continuous derivations of completed Lie algebras which send, respectively,  $\x_1$ to $0$ or $[a,b]$ to $0$.
 We say that an element $\sigma \in  \Der^1 \LL(\x_0,\x_1)^{\wedge}$ lifts to a derivation $\widetilde{\sigma} \in \Der^{\Theta} \, \LL(a,b)^{\wedge}$ 
 (and conversely, $\widetilde{\sigma}$ descends to the derivation $\sigma$) if 
\begin{equation} 
\widetilde{\sigma}\circ    \phi  = \phi \circ  \sigma \  . \
\end{equation} 
which  is equivalent to the equation $\widetilde{\sigma}    \phi(\x_0)   = \phi   \sigma(\x_0)$.  
An element of $\Der^{\Theta} \, \LL(a,b)^{\wedge}$ descends  to an element of $   \Der^1 \LL(\x_0,\x_1)^{\wedge}$  if and only if it preserves the subspace  $\phi(\LL(\x_0,\x_1)^{\wedge})$ in $\LL(a,b)^{\wedge}$.
Since  $(\ref{phidefn})$ is injective, $\sigma$ and $\widetilde{\sigma}$ determine each other uniquely. Furthermore,  lifting  derivations preserves composition:  if 
$\sigma_1, \sigma_2 \in \Der^1 \, \LL(\x_0,\x_1)^{\wedge}$ admit lifts $\widetilde{\sigma}_1, \widetilde{\sigma}_2 \in \Der^{\Theta} \, \LL(a,b)^{\wedge}$, then   $[\sigma_1, \sigma_2]$  admits the lift $[\widetilde{\sigma}_1, \widetilde{\sigma}_2]$.

Since $(\ref{P1toEtimes})$ is geometric, it is compatible with the actions of $G^{dR}$ on the de Rham fundamental groups (the case of the Hodge realisation is \cite{HaDe}, Theorem 15.1). 
Thus $\phi$ commutes with the action of $\g$, and the morphism $(\ref{i1def})$
is the lift of the map $i_0: \g \rightarrow \Der^1\,  \LL(\x_0,\x_1)$. In particular,
\begin{equation} \label{i1sigmaequation} 
 i_1(\sigma)  \big(   \phi( \x_0) \big) = \phi(    i_0(\sigma) (\x_0) ) \qquad \hbox{ for all } \sigma \in \g\ , 
\end{equation}
where $i_0$ was defined in   $(\ref{gfraktoL})$. The injectivity of $i$  \cite{BrMTZ} and $\phi$ implies the 
\begin{thm} The map $i_1:\g \rightarrow \Der^{\Theta} \, \LL(a,b)$ is injective.
\end{thm}

\subsection{$B$-filtration} \label{sectBdegree}  One possible  way to cut out the depth filtration on the image of $\g$ inside $\Der^{\Theta}\, \LL(a,b)$
is via the $B$-filtration on $\PP$.  As pointed out by Hain, it can be defined as the convolution of the Hodge filtration 
$F$ and the lower central series filtration $L$:
$$B^r = (F \star L)^r = \sum_{a+b=r} F^a \cap L^b\  .$$
This induces a decreasing filtration on the Lie algebra of $\PP$.  Note that $L^b= W_{-b}$, where $W$ denotes the weight filtration.
Since the filtrations $F,W,$ and $M$ can be split simultaneously \cite{HaDe}, $B$ induces a filtration on $\LL(H)= \gr^W \gr^M \Lie\, \PP$, and 
$$B^r \LL(H) = \{ w \in \LL(a,b) :  \deg_b w \geq r\}\ , $$
is the filtration associated to the 
$B$-degree, where we define the $B$-degree on $\LL(a,b)$ to be the degree in $b$. 
 To see this, note that $B^r \gr^n_L \LL(H_{dR})= F^{n-r} \gr^n_L \LL(H_{dR})$. Since $\gr^n_L \LL(H_{dR})$ consists of words of length $n$ in $a$ and $b$, 
 and since $$\Q a \oplus \Q b=F^{-1}H_{dR} \supset F^0 H_{dR} = \Q b \ ,$$
 $F^{n-r}   \gr^n_L \LL(H_{dR})$ is spanned by words with at least $r$ letter $b$'s.
   Note also that $$2 \deg_b +2 \deg_W= \deg_M \ .$$  
    The $B$-filtration also induces a decreasing filtration $B^{\bullet}$ on $\Der^{\Theta}\,   \LL(a,b)$, such that 
    $\Der^{\Theta}\,   \LL(a,b) = B^{-1} \Der^{\Theta}\,   \LL(a,b)$.
     It is the  decreasing filtration associated to the
   grading induced by the $B$-degree.  The subspace $B^0 \, \Der^{\Theta}\,   \LL(a,b)$ is the space of derivations $\delta$ such that 
   $\delta(b) \subset B^1 \, \LL(a,b)$, i.e., such that the coefficient of $a$ in  $\delta (b)$ is zero.
   \begin{lem}  \label{lemBmotivic} The $B$-filtration on $\PP$ is motivic, i.e., it is stable under the image of the de Rham motivic Galois group  $i_1(\g)$. 
    \end{lem} 
   \begin{proof}   We must verify that  $i_1(\g) \subset B^0 \, \Der^{\Theta}\, \LL(a,b)$. 
Since the group $G^{dR}$ acts on  $\LL(a,b)^{ab} = H_{dR}$
 through its quotient $\G_m$, and because  $H_{dR} = \Q \oplus \Q(1)$ is a direct sum of pure Tate motives, the graded Lie algebra of its pro-unipotent radical $\g$ acts trivially on $
 \LL(H_{dR})^{ab}$. Therefore  $i_1(\g)(b) \subset [\LL(H_{dR}), \LL(H_{dR})] \subset B^1 \LL(H_{dR})$. 
      \end{proof}

\begin{lem} 
Let  $\alpha \in \LL(\x_0,\x_1)^{\wedge}$. Then $\alpha \in D^r$ if and only if $\phi(\alpha) \in B^r$.
\end{lem}
\begin{proof} The fact that $\phi \, D^r \subset B^r$ is clear from the definition  $(\ref{phidefn})$. The converse follows from the fact that 
 the associated graded  morphism
$$\phi^0 : \gr^{\bullet}_D \LL(\x_0,\x_1) \To   \gr^{\bullet}_B \LL(a,b)$$
given by $\phi^0(\x_0)=a$ and $\phi^0(\x_1)=[a,b]$,  is injective.
\end{proof} 
Observe that if $\delta \in \Der^1 \, \LL(\x_0,\x_1)$ then $\delta \in D^r$ if and only if $\delta(\x_0) \in D^r$.
\begin{lem}   Let $\delta \in B^0 \, \Der^{\Theta}\,   \LL(a,b)$.  The following are equivalent
\begin{enumerate}
\item $ \delta \in B^r  \, \Der^{\Theta}\,   \LL(a,b)$
\item $\delta(a) \in B^r \,  \LL(a,b)$
\item $\delta(\phi(\x_0)) \in B^r\, \LL(a,b)$\ .
\end{enumerate}
\end{lem}
\begin{proof}
Clearly $(1)$ implies $(2)$ and $(3)$.  Now suppose that $(2)$ holds. We have
$$0 = \delta[a,b] = [a, \delta(b)]+ [\delta(a), b]$$
which implies that $[a,\delta(b)] \in B^{r+1} \, \LL(a,b)$.  Since the coefficient of $a$ in $\delta(b)$ vanishes, this implies that $\delta(b) \in B^{r+1} \LL(a,b)$. Together with $\delta(a) \in B^r$ this implies $(1)$.
 
 Now suppose $(3)$ holds. Write $\phi(\x_0)= a +w$ where $w\in B^1 \LL(a,b)$, and by $(3)$, 
   $$\delta(a) +\delta(w) \in B^r\ .$$
  If we have shown that  $\delta(a) \in B^i$, then by $(2) \Rightarrow (1)$, we have $\delta \in B^i$ and hence  $\delta(w) \in B^{i+1}$.  The previous equation
  then implies $\delta(a) \in B^{\min(i+1,r)}$. Starting with $i=0$, repeat this argument to deduce that $\delta \in B^r$ which proves $(3) \Rightarrow (1)$.
\end{proof} 
\begin{prop}  Let $\sigma \in \Der^1 \LL(\x_0,\x_1)$ which lifts to an element $\widetilde{\sigma} \in B^0\, \Der^{\Theta} \, \LL(a,b)$. Then
$ \sigma \in D^r$ if and only if $\widetilde{\sigma} \in B^r$.
\end{prop}
\begin{proof} We have
$\phi( \sigma(\x_0)) = \widetilde{\sigma} ( \phi(\x_0)).$
Now apply the previous two lemmas to deduce that $\sigma\in D^r$ if and only if $\widetilde{\sigma} \in B^r$.
\end{proof}
 \begin{cor} \label{corBcutout} The $B$-filtration cuts out the depth filtration on the image of $\g$:  \begin{equation} 
B^r  \cap i_1(\g)=  i_1(D^{r } \g)   \ . 
 \end{equation} 
 \end{cor}
 \begin{proof} Apply lemma $\ref{lemBmotivic}$, the previous proposition, and
 $i_1 \phi = \phi \, i_0$ $(\ref{i1sigmaequation})$.  \end{proof}
By considering the symmetry $t\mapsto 1-t$ on $\Pro^1\backslash \{0,1,\infty\}$, one knows that an element $\sigma \in \g$ of degree $-m$ is uniquely determined by its image in $D^1 / D^{\lfloor \! {m\over 2} \!\rfloor}  \LL(\x_0,\x_1) $ (duality relation). It follows from the injectivity of $\phi$ and $i$ that:

\begin{cor}  
An element  $\sigma \in \g$ of degree $-m$ is uniquely determined by the image of  $i_1(\sigma)$ in 
$ B^1 / B^{\lfloor \! {m\over 2} \!\rfloor}   \LL(a,b)\ .$
\end{cor}
\begin{rem}
Thus the `tails' of the elements $i_1(\sigma)$ in the $B$-filtration, which may be infinitely long, are uniquely determined from their `heads' $i_1(\sigma) \pmod { B^{\lfloor \! {m\over 2} \!\rfloor} }$. Furthermore, one can show (for instance, using the description \cite{MMV}, \S9 of the group  of 
automorphisms of a semi-direct product) that, modulo $B^{m}$, $i_1(\sigma)$ is equivalent to an element of $\ue$, so can be expressed (non-uniquely) in terms of the geometric derivations $\varepsilon_{2n+2}^{\vee}$ for $n\geq -1$. For instance, a choice of zeta element of degree $-2n-1$ admits an expansion
$$i_1(\sigma_{2n+1}) \equiv  \sum_{a_1,\ldots, a_r \in 2 \N}  \lambda^{(2n+1)}_{a_1,\ldots, a_r}  \big[ \varepsilon^{\vee}_{a_1}, \big[\varepsilon^{\vee}_{a_2},\big[ \ldots , \big[ \varepsilon^{\vee}_{a_{r-1}}, \varepsilon^{\vee}_{a_{r}}\big]\cdots \big] \pmod{B^{2n+1}}\  $$
where $\lambda^{(2n+1)}_{a_1,\ldots, a_r}\in \Q$.
 The non-uniqueness of this expansion corresponds to the fact that there exist relations between the generators $\varepsilon_{2n+2}^{\vee}$ in $\ue$ \cite{Po}. Nonetheless, this geometric expansion or `anatomy' of the elements $i_1(\sigma^c_{2n+1})$, for $n\geq 2$, will be determined explicitly modulo $B^4$ below. \end{rem}


\section{Commutative power series} \label{sectCommSeries}
We discuss commutative power series representations for the Lie algebras $\LL(\x_0,\x_1)$, $\LL(a,b)$ and 
describe the composition laws for their derivation algebras.
\subsection{Commutative power series} \label{algtopoly} The method of commutative power series
is based on  the observation that there is an isomorphism of $\Q$-vector spaces, for $r\geq 0$
\begin{eqnarray} \label{rhodef}
\rho \quad : \quad \gr^r_D T(\x_0,\x_1)   & \overset{\sim}{\To} & \Q[y_0,\ldots, y_r] \\
\x_0^{i_0} \x_1\x_0^{i_1} \ldots  \x_1 \x_0^{i_r}  & \mapsto & y_0^{i_0} \ldots y_r^{i_r} \nonumber 
\end{eqnarray}
Let us denote by 
$$ P =    \bigoplus_{r\geq 0} \Q[y_0,\ldots, y_r]\ . $$
Since $\LL(\x_0,\x_1)$ is $D$-graded, $(\ref{rhodef})$  induces a map $\rho : \LL(\x_0,\x_1) \rightarrow P$.  One checks that
\begin{equation}  \label{rhoofadx0x1}
\rho( \mathrm{ad}(\x_0)^{2n}(\x_1 ) ) =  (y_0-y_1)^{2n} \ .
\end{equation}
Furthermore, one  shows (\cite{BrDepth}, lemma 6.2) that the image of $\gr^r_D \LL(\x_0,\x_1)$, for $r\geq 1$, is contained in the 
subspace of polynomials which are translation invariant:
$$f(y_0,\ldots, y_r) = f(y_0+\lambda, \ldots, y_r+ \lambda)  \quad \hbox{ for all }  \quad \lambda \in \Q\ .$$
Such a polynomial $f$ is uniquely determined by its image in
\begin{eqnarray}  \label{transmap}
 \Q[y_0,\ldots, y_r ]  &  \To &   \Q[x_1,\ldots, x_r]   \\ 
\qquad \qquad f(y_0,\ldots, y_r)   & \mapsto &   \overline{f}(x_1,\ldots, x_r) = f(0,x_1,\ldots, x_r)\ . \nonumber
\end{eqnarray} 
We sometimes call this the reduced representation of a translation-invariant polynomial, and it applies equally well to translation-invariant rational functions. 
In this way, since $\LL(\x_0,\x_1) \cong \gr^{\bullet}_D \LL(\x_0,\x_1)$ is  graded  with respect to the $D$-degree,  every element of $D^1\LL(\x_0,\x_1)$ can be uniquely written
\begin{eqnarray} \label{LtoQ}
\overline{\rho}\quad : D^1 \LL(\x_0,\x_1)  & \To &  \bigoplus_{r\geq 1} \Q[x_1,\ldots, x_r]       \\
\sigma & \mapsto & \sum_{r\geq 1} \sigma^{(r)}    \nonumber 
\end{eqnarray}
where $\sigma^{(r)}$ denotes the polynomial representation of the component in depth $r$ of $\sigma$.
The zeta elements satisfy $\sigma_{2n+1}^{(1)} = x_1^{2n}$, for all $n \geq 1$ by $(\ref{rhoofadx0x1})$.
\subsection{Concatenation products}
The  concatenation of words in the alphabet $\x_0,\x_1$  defines a non-commutative multiplication law $f,g \mapsto f\cdot g$:
$$\gr^{r}_D \LL(\x_0,\x_1) \times \gr^{s}_D \LL(\x_0,\x_1) \To \gr^{r+s}_D \LL(\x_0,\x_1)\ .$$
On the level of commutative power series, it is the operation
\begin{eqnarray} \label{cdotdefn}
 \Q[y_0,\ldots, y_r] \otimes \Q[y_0,\ldots, y_{r+s} ] & \To&  \Q[y_0,\ldots, y_{r+s} ] \\
f(y_0,\ldots, y_r) \cdot  g(y_0,\ldots, y_{r+s}) & =  & f(y_0,\ldots, y_r) g(y_r, \ldots, y_{r+s}) \nonumber
\end{eqnarray} 
It follows from the definition of the linearized Ihara action $(\ref{circbdef})$ that 
\begin{eqnarray} \label{Distrib}
f \circb (g\cdot h) = (f \circb g) \cdot h + g\cdot (f \circb h) - g \cdot f \cdot h  \  .
\end{eqnarray} 
Note that equation $(\ref{Distrib})$ is equivalent to the condition that  the linear map 
$$ w \mapsto f\circb w - w \cdot f$$ is a derivation with respect to $\cdot$.
There is another concatenation product, denoted by $\studot$,  which comes from the stuffle Hopf algebra \cite{AA}. It will only be used once in this paper, so will not be discussed in any detail here.

\subsection{Linearized Ihara action} \label{sectIharaonpoly}
The  operator $\circb: T(\x_0,\x_1)  \otimes_{\Q} T(\x_0,\x_1)  \rightarrow T(\x_0,\x_1)$  
is   homogeneous for the $D$-degree, and therefore  defines a map
\begin{eqnarray} \circb: \Q[y_0,\ldots, y_r] \otimes_{\Q} \Q[y_0,\ldots, y_s]   & \To & \Q[y_0,\ldots, y_{r+s}]   \\
f(y_0,\ldots, y_r) \otimes g(y_0,\ldots, y_s)  & \mapsto &  f\circb g\, ( y_0,\ldots, y_{r+s}) \nonumber
\end{eqnarray}
whose $r,s$ component  is given explicitly by the  formula
\begin{multline}\label{circformula}
f \circb g \, (y_0,\ldots, y_{r+s})  =  \sum_{i=0}^s f(y_i,y_{i+1}, \ldots, y_{i+r}) g(y_0,\ldots, y_i, y_{i+r+1}, \ldots, y_{r+s}) + \\
  (-1)^{\deg f + r} \sum_{i=1}^s f(y_{i+r},\ldots, y_{i+1},y_i) g(y_0,\ldots, y_{i-1}, y_{i+r}, \ldots, y_{r+s}) \ . 
 \end{multline}
 This  can be read off  from equation $(\ref{circbdef})$.  Antisymmetrizing gives a pairing
 $$\{ f , g\} = f\circb g - g\circb b$$
 which coincides with the Ihara bracket on the image of $\LL(\x_0,\x_1)$. 
Clearly, if $f,g$ are both translation invariant, then so too are $f\circb g$ and $\{f,g\}$.
\begin{example} \label{exampleofIharaactionlength2} When $r=s=1$ and $f,g \in \Q[y_0,y_1]$ are translation-invariant, one verifies that the image of  $\{f,g\}\in \Q[y_0,y_1,y_2]$  in $\Q[x_1,x_2]$ is given by 
$$ \overline{f}(x_1)\overline{g}(x_2)
-\overline{g}(x_1)\overline{f}(x_2)+ \overline{f}(x_2-x_1) (\overline{g}(x_1) -\overline{g}(x_2)) + 
(\overline{f}(x_2)-\overline{f}(x_1))\overline{g}(x_2-x_1) \ .$$

\end{example}
The Ihara action $(\ref{circformula})$ extends to rational functions by the identical formula.

\subsection{Derivations on $\LL(\x_0,\x_1)$ and  power series} Consider the isomorphism
\begin{eqnarray} \label{deltadefn}
\LL(\x_0,\x_1)  & \overset{\sim}{\To} & \Der^1 \, \LL(\x_0,\x_1)  \\
w & \mapsto & \delta_w \nonumber
\end{eqnarray} 
where $\delta_w$ denotes the derivation $\delta_w(\x_0) =w$, $\delta_w(\x_1)=0$. This isomorphism respects the $D$-grading on both sides.
 Let 
$$P' =  \bigoplus_{r\geq 1} \Q[y_0,\ldots, y_r] {1 \over y_0-y_r}\ . $$
Using $(\ref{rhodef})$, define a linear map
\begin{eqnarray} \label{rhoprimedefn}
\rho': D^1 \Der^1\, \LL (\x_0, \x_1)  &\To &  P'  \\ 
\delta_w & \mapsto &  \sum_{r \geq 1} \rho^r(w) {1 \over y_0-y_r}\ .  \nonumber
\end{eqnarray}

\begin{prop}   \label{propAction1} The following diagram commutes:
$$ 
\begin{array}{cccccc}
 & D^1\Der^1 \LL(\x_0,\x_1) &  \times& \LL(\x_0,\x_1)  & \To    & \LL(\x_0,\x_1)  \\
 &\downarrow_{\rho'}   &&\downarrow_{\rho} &   &  \downarrow_{\rho}  \\
\odot :  &    P'   &\times& P  &  \To   & P    
\end{array}
$$
where  for $f\in P'$ and $g\in P$, 
\begin{equation} \label{odotdefn}
f\odot g =  f \circb g - f\cdot g  \ .
\end{equation}
Furthermore, the composition of derivations is given by the linearised Ihara bracket:
\begin{equation} \label{alphabracketeqn}
\rho'  ([ \delta_w, \delta_v] )=  \{\rho'(\delta_w), \rho'(\delta_v)\} \ .
\end{equation} 
Thus $\rho': D^1 \Der^1 \LL(\x_0,\x_1) \rightarrow (P', \{ , \})$ is a morphism of Lie algebras. 
\end{prop}
\begin{proof}
The shuffle distributivity law $(\ref{Distrib})$, and the remark which follows, implies that  $\odot$ is a derivation:
$f \odot (g \cdot h) = (f\odot g) \cdot h +  g\cdot ( f \odot h)$, for all $f \in P'$, $g,h \in P$.
It therefore suffices to show that for all $w\in \LL(\x_0,\x_1)$, and $i=0,1$, we have 
$$\rho'(\delta_w) \odot \rho(\x_i) = \rho(\delta_w(\x_i))\ .$$
Without loss of generality, assume that $w$ is of $D$-degree $r \geq 1$, and write $f= \rho'(w)$. Since $w\in \LL(\x_0,\x_1)$, we have
$w+w^*=0$ and hence 
$$f(y_0,\ldots, y_r) + (-1)^{\deg f+r} f(y_r,\ldots, y_0) =0\ .$$ 
We check that  $\rho(\x_0)= y_0 \in \Q[y_0]$ and $\rho(\x_1)=1 \in \Q[y_0,y_1]$, and verify that
\begin{eqnarray} 
 f(y_0,\ldots, y_r) \circb y_0  & = & y_0 f(y_0,\ldots, y_r) \nonumber \\
 f(y_0,\ldots, y_r) \circb 1  & =  &  f(y_0,\ldots, y_r) \nonumber
\end{eqnarray} 
directly from the definition $(\ref{circformula})$, applied in the cases $s=0, g=y_0$ and $s=1, g=1$ respectively. 
Via $(\ref{odotdefn})$, these equations imply that $f \odot y_0 = (y_0-y_r) f$ and $f\odot 1 = 0$.
This proves that $\rho'(\delta_w) \odot \rho(\x_0)= \rho(w)$  and $\rho'(\delta_w) \odot \rho(\x_1)=0 $ as required. For the last part,
use the fact that  $\Der^1 \LL(\x_0,\x_1)$ acts faithfully on $\LL(\x_0,\x_1)$  and the identity
$$ f\odot (g \odot h) - g \odot (f \odot h) = \{f,g\} \odot h$$
which follows from  $(\ref{odotdefn})$, $(\ref{Distrib})$ and $(\ref{Assoc})$. \end{proof}

The following corollary is not required for the remainder of this paper.
\begin{cor} \label{cordx1} Let  $w \in \LL(\x_0,\x_1)$ of $D$-degree $r$, and write $f=\overline{\rho}^{r}(w)$. 
Then
\begin{equation}
\overline{\rho}^{(r+1)}( \delta_{\x_1} w) =   \{ fx_r^{-1} , x_1^{-1}\}  x_{r+1} \quad \in \quad  \Q[x_1,\ldots, x_{r+1}]
\end{equation} 
\end{cor}
\begin{proof} 
Use $(\ref{deltadefn})$ and $(\ref{rhoprimedefn})$. Apply $(\ref{alphabracketeqn})$ to   $[ \delta_{\x_1}, \delta_w] = \delta_{\delta_{\x_1}w} $ to give  $$\rho' (\delta_{\delta_{\x_1}w }) = \{ \rho' (\delta_{\x_1}), \rho' (\delta_w) \}\ .$$
The left-hand side is $\rho^{(r+1)}(\delta_{\x_1}w) \over (y_0-y_{r+1})$,  the right-hand side is $\{{1 \over y_0-y_1}, {\rho^{(r)}(w) \over y_0-y_r}\}$. 
Then pass to the reduced representation  $(y_0,y_1, \ldots, y_r)\mapsto (0,x_1,\ldots, x_r)$.
\end{proof}
\subsection{Derivations on $\LL(a,b)$ and  power series}  \label{sectLabalgtopoly}
Define
$$D^{\Theta} = \Der^{\Theta}(\LL(a,[a,b]) , B^1 \LL(a,b))$$
to be the vector space of linear maps $\delta: \LL(a,[a,b]) \rightarrow B^1\LL(a,b)$ satisfying 
$\delta[p,q]=[\delta(p),q]+[p,\delta(q)]$ and $\delta([a,b])=0$. Such a $\delta \in D^{\Theta}$ is uniquely determined by the element 
$\delta(a)\in B^1 \LL(a,b)$.
Furthermore, there is an injective map
$$ B^1 \Der^{\Theta}(\LL(a,b)) \To D^{\Theta}\  $$
obtained by restricting to the Lie subalgebra $\LL(a,[a,b]) \subset \LL(a,b)$.

 Denote also by $\rho$ the linear map
\begin{eqnarray} \label{rhosecond}
\rho^{(r)} \quad : \quad \gr^r_B\,  T(a,b)   & \overset{\sim}{\To} & \Q[y_0,\ldots, y_r] \\
a^{i_0} b a^{i_1} \ldots  b a^{i_r}  & \mapsto & y_0^{i_0} \ldots y_r^{i_r}\  . \nonumber 
\end{eqnarray} 
Let  $\ell_0=1$ and for $r \geq 1$,  set 
 \begin{equation} \label{ellrdefn} 
 \ell_r = (y_0-y_1)(y_1-y_2) \ldots (y_{r-1}-y_{r})\ .
 \end{equation}   Define a  graded vector space 
$$ Q= \bigoplus_{r\geq 0}  \ell_r^{-1}  \Q[y_0, \ldots, y_r]  $$
and a linear map
\begin{eqnarray}
\label{abtorat}
\ell:   \LL(a,b)   & \To &  Q  \\
  w & \mapsto &  \sum_{r} \ell_r^{-1} \rho^{(r)}(w)  \nonumber
\end{eqnarray}
Since $\ell_r \cdot  \ell_s = \ell_{r+s}$, $Q$ is an algebra for the shuffle concatenation $(\ref{cdotdefn})$.
Define 
$$Q' = \bigoplus_{r\geq 1} Q_{r} {1 \over y_0-y_r}$$
and setting $c_r = \ell_r(y_0-y_r)$ ($c$ for `cyclic') consider the linear map 
\begin{eqnarray}  \label{ellprimedefn}
\ell': D^{\Theta}  & \To&  Q'  \\
\delta & \mapsto &      \sum_{r\geq 1} c_r^{-1} \rho^{(r)} \delta(a) \ .   \nonumber 
\end{eqnarray} 
It is injective, since $\delta(a)$ uniquely determines $\delta \in  D^{\Theta}$.
 \begin{prop} \label{propAction2} The following diagram commutes:
$$ 
\begin{array}{cccccc}
  &  D^{\Theta}  &  \times& \LL(a,[a,b])  & \To    & \LL(a,b)  \\
 &\downarrow_{\ell'}  &&\downarrow_{\ell} &   &  \downarrow_{\ell}  \\
\eact :  &    Q'   &\times& Q  &  \To   & Q    
\end{array}
$$
where, as in     $(\ref{odotdefn})$, we have
\begin{equation} \label{eactdefn}
f\eact g = f\circb g - f\cdot g \  . 
\end{equation} 
Restricting to the subspace $ B^1 \Der^{\Theta} \LL(a,b)\subset D^{\Theta}$ we obtain a commutative diagram
$$ 
\begin{array}{cccccc}
  &  B^1 \Der^{\Theta} \LL(a,b) &  \times& \LL(a,b)  & \To    & \LL(a,b)  \\
 &\downarrow_{\ell'}  &&\downarrow_{\ell} &   &  \downarrow_{\ell}  \\
\eact :  &    Q'   &\times& Q  &  \To   & Q    
\end{array}
$$
Furthermore, we have the identity for all $\delta_1, \delta_2 \in B^1 \Der^{\Theta} \LL(a,b)$:
\begin{equation} \label{betabracketeqn}
\ell'  ([ \delta_1, \delta_2] )=  \{\ell'(\delta_1), \ell'(\delta_2)\} \ .
\end{equation} 
Thus $\ell': B^1 \Der^{\Theta} \LL(a,b) \rightarrow (Q', \{ , \})$ is a Lie algebra homomorphism.
 \end{prop} 

\begin{proof} The proof is similar to the proof of proposition $\ref{propAction1}$, except that we must check 
that $\ell'(\delta) \eact \ell([a,b])= 0$ and $\ell'(\delta) \eact \ell(a)= \ell(\delta(a)).$ But $\ell([a,b])= {y_0-y_1 \over y_0-y_1}=1$ and $\ell(a)= y_0$, so this 
calculation is formally identical to the one in proposition $\ref{propAction1}$.
\end{proof}

\subsection{Double shuffle equations}The equations defining $\dmr_0$ can be spelt out explicitly and translated  via $(\ref{algtopoly})$ into the language of commutative power series \cite{AA}. 
 We shall only require their restriction 
to depths $\leq 3$ and work with translation-invariant representations $(\ref{transmap})$.  Let 
$$  (f^{(1)}, f^{(2)}, f^{(3)}) \quad \in \quad  \Q[x_1]\oplus \Q[x_1,x_2] \oplus \Q[x_1,x_2,x_3]$$
Writing $x_{ij}$ for $x_i+x_j$ and $x_{ijk}$ for $x_i+x_j+x_k$, the shuffle equations modulo products in depths $2$ and $3$ are 
given by 
\begin{eqnarray} \label{shuffle}
f^{(2)}(x_1,x_{12}) + f^{(2)}(x_2,x_{12}) & = & 0 \\ 
f^{(3)}(x_1, x_{12}, x_{123})+f^{(3)}(x_2, x_{12}, x_{123})+f^{(3)}(x_2, x_{23}, x_{123}) & = & 0\ .  \nonumber \
\end{eqnarray}
The solutions to the shuffle equations modulo products correspond, via $(\ref{rhodef})$, to  the image of  $\LL(\x_0,\x_1) $ inside $T(\x_0,\x_1)$. The stuffle equations modulo products,  in depths $2$ and $3$, correspond to the (regularised versions of) the equations ($a,b,c\in \N$):
$$\zeta(a, b) + \zeta(b,a) + \zeta(a+b) \equiv 0 \pmod{\hbox{products}}$$
$$\zeta(a, b,c) + \zeta(b,a,c) + \zeta(b,c,a) +\zeta(a+b,c)+ \zeta(a,b+c) \equiv 0 \pmod{\hbox{products}}\ .$$
By considering the  series $Z^{(r)}= \sum_{n_1,\ldots, n_r\geq 0} \zeta_{*}(n_1,\ldots, n_r) x_1^{n_1-1} \ldots x_r^{n_r-1}$, 
where the subscript $*$ denotes the stuffle regularisation, the previous equations translate into 
\begin{equation}  \label{stuffle}
f^{(2)}(x_1, x_2)+f^{(2)}(x_2, x_1) \quad  = \quad    { f^{(1)}(x_1) - f^{(1)}(x_2) \over  x_2-x_1 }   
\end{equation}
\begin{multline}
f^{(3)}(x_1, x_2, x_3)+f^{(3)}(x_2, x_1, x_3)+f^{(3)}(x_2, x_3, x_1) \quad   =   \quad   \nonumber \\
   { f^{(2)}(x_2, x_1)-f^{(2)}(x_2,x_3)  \over  x_3-x_1}       +  { f^{(2)}(x_1, x_3)-f^{(2)}(x_2,x_3) \over  x_2-x_1}    \nonumber
\end{multline}
Note that the right-hand sides of the equations are in fact polynomials. These equations extend to an infinite family of equations
in every depth \cite{AA}. The Lie algebra $\dmr_0$ is defined to be the sets of solutions to both shuffle and stuffle equations.

The linearised double shuffle equations  \cite{BrDepth} are the same sets of equations in which the right-hand sides are zero. The linearised shuffle equations are identical to the ordinary shuffle equations, but the linearised stuffle equations are:
\begin{eqnarray}  \label{linstuffle}
f^{(1)}(x_1)+f^{(1)}(-x_1)\quad = \quad 0 \\ 
f^{(2)}(x_1, x_2)+f^{(2)}(x_2, x_1) \quad  = \quad   0   \nonumber \\
f^{(3)}(x_1, x_2, x_3)+f^{(3)}(x_2, x_1, x_3)+f^{(3)}(x_2, x_3, x_1) \quad   =    \quad0   \nonumber
\end{eqnarray}
 The linearised double shuffle equations are also closed under the Ihara bracket, by a (simpler) version of Racinet's theorem.

\subsection{Geometric derivations and linearised double shuffle equations with poles.} 
The Lie algebra $\ue \subset B^1 \Der^{\Theta} \LL(a,b)$  of geometric derivations was defined in $(\ref{uedefn})$.
The definition of the linearised double shuffle equations $(\ref{shuffle})$ and $(\ref{linstuffle})$ can be extended in the obvious way to rational functions. 

\begin{defn} \label{defnpls} Define $\mathfrak{pls} \subset Q'$  to be the subspace of $Q'$ satisfying the linearised double shuffle equations.  It is bigraded by weight and depth.
\end{defn} 
By a version of Racinet's theorem,  $\mathfrak{pls}$ is also a Lie subalgebra of $Q'$ for the linearised Ihara bracket $\{ \ , \}$. The notation stands for `polar linearised double shuffle' solutions. It is a  bigraded Lie algebra in the category of $\mathfrak{sl}_2$-representations over $\Q$.

\begin{prop}  \label{propuedbshf} The geometric derivations, in their reduced rational function representation $(\ref{ellprimedefn})$, satisfy the linearised double shuffle equations:
 $$\overline{\ell'}(\ue) \subset \mathfrak{pls} \ . $$
\end{prop} 
\begin{proof} The images of the generators $\overline{\ell'}( \varepsilon^{\vee}_{2n+2}) = x_1^{2n}$ by $(\ref{rhoofadx0x1})$ for $n\geq -1$.  They  are even and hence  solutions to the linearised double shuffle
equations. It follows from  $(\ref{betabracketeqn})$, that $\overline{\ell'}$ is a morphism of Lie algebras since $\mathfrak{pls}$ is closed under $\{ \ , \ \}$.
\end{proof}

 A natural question to  ask is whether
$\mathfrak{pls} = \ue$. It is  true in depths $\leq 3$  and in certain limits \cite{AA}.

The previous proposition implies that the stuffle equations define maps from the space of non-geometric derivations
$$\Der^{\Theta}\, \LL(a,b) / \ue $$
to spaces of rational functions.  We shall show in remark \ref{remz_3} that this map is non-zero and  provides a tool to 
prove that certain derivations are not geometric. This can in particular be applied to the  image of  $i_1(\g)$ in $\Der^{\Theta}\, \LL(a,b) / \ue $.

\section{Zeta elements in depth 3 via anatomical construction} \label{sectDrin3}
We wish to write down elements 
$$\sigma^c_{2n+1} \in  D^1 \LL(\x_0, \x_1) / D^4 \LL(\x_0,\x_1)$$
by exhibiting explicit polynomials
$$(\sigma_{2n+1}^{(1)}, \sigma_{2n+1}^{(2)},\sigma_{2n+1}^{(3)}) \in  \Q[x_1] \oplus \Q[x_1,x_2] \oplus \Q[x_1,x_2,x_3]$$
which are solutions to the equations $(\ref{shuffle})$ and $(\ref{stuffle})$.

\subsection{Polar solutions}  \label{sectChis}The shape of the equations $(\ref{stuffle})$ suggests searching for solutions amongst the space of rational functions in $x_i$ with $\Q$-coefficients. 
Let 
\begin{equation} \label{sdef}
s^{(1)} = {1 \over  2 \, x_1} \qquad \hbox{ and } \qquad s^{(2)} ={1\over 12} \Big(  {1 \over  x_1x_2 } + {1 \over  x_2(x_1-x_2)}\Big)\ .
\end{equation}
It is easy to verify that $(s^{(1)}, s^{(2)})$ is a solution to the  double shuffle equations $(\ref{shuffle})$ and $(\ref{stuffle})$ in depths one and two. 

\begin{defn} For $n\geq -1$, define rational functions in $x_1,x_2, x_3$ by
\begin{eqnarray}
\xi_{2n+1}^{(1)} & = &  x_1^{2n}  \\ 
\xi_{2n+1}^{(2)}  & =  &  \{  s^{(1)} , x_1^{2n}\} \nonumber \\ 
\xi_{2n+1}^{(3)} &  =  &  \{s^{(2)}, x_1^{2n}\}  + {1 \over 2 } \{ s^{(1)}, \{ s^{(1)} , x_1^{2n}\}\}   \nonumber 
\end{eqnarray} 
where curly brackets denote the linearised Ihara bracket. Explicitly, we have
\begin{equation} \label{xiindepth2}
\xi_{2n+1}^{(2)} = {x^{2n}_{2} -(x_{2}-x_1)^{2n} \over x_{1}}+{x^{2n}_{1} - x^{2n}_{2} \over x_{2}-x_{1}}+{ (x_{2}-x_{1})^{2n}  -x^{2n}_{1}\over x_{2}}   
\end{equation} 
which defines a polynomial in $\Q[x_1,x_2]$ whenever $n\geq 0$. On the other hand, 
$\xi_{2n+1}^{(3)} $ is a rational function in $x_1,x_2,x_3$ with non-trivial poles. 
When $n \geq 0$ it has at most simple poles along
$x_1=0, x_3=0, x_1=x_2$ and $x_2= x_3 $. 

One checks that the case $n=0$ is trivial: $\xi_1^{(1)}=1$ and $\xi_1^{(2)}=\xi_1^{(3)}=0$.
\end{defn}

\begin{prop} Let $n\geq -1$. The elements $$\xi_{2n+1}= (\xi^{(1)}_{2n+1}, \xi^{(2)}_{2n+1}, \xi^{(3)}_{2n+1})$$ satisfy the double shuffle equations modulo products $(\ref{shuffle})$ and $(\ref{stuffle})$ in depths $2,3$.
\end{prop} 
\begin{proof} This is a finite computation and only uses the fact that $x_1^{2n}$ is an even function. It also follows from the fact
that $(s^{(1)}, s^{(2)})$, and $x_1^{2n}$  are solutions to the double shuffle equations via  a version of Racinet's theorem for rational functions.
\end{proof} 
\begin{rem}  \label{rem extend} The elements $\xi_{2n+1}$ can be extended to all higher depths by the equation 
$$ \xi_{2n+1} = \exp(\ad(s)) x_1^{2n}$$ to define 
solutions to the full set of double shuffle equations with poles, where $s$ is one of (many possible) solutions to  the polar double shuffle equations in weight $0$. This   is discussed in \cite{AA}, and implies the  previous proposition.
The component of $s$ in depth 3 is unique by an extension of the  depth-parity theorem for double shuffle equations (theorem $\ref{Depth-Paritythm}$) to the case of rational functions and  given by 
$s^{(3)} = {1 \over 2} \{s^{(1)}, s^{(2)}\}\ .$
 \end{rem}

\subsection{Definition of canonical elements}
It is convenient to define heretical normalisations of the elements $\xi$  as follows. Let
\begin{equation}  \underline{\xi}_{-1}      =     {1 \over 12}   \,\xi_{-1}  \qquad \hbox{ and } \qquad 
 \underline{\xi}_{2n+1}     =     {B_{2n} \over (2n)!} \,  \xi_{2n+1}  \quad \hbox{for }  n\geq 0 
 \end{equation} 
where $B_{2n}$ is the $2n^{\mathrm{th}}$ Bernoulli number.
 Set 
\begin{equation} \label{bdef}
b(x) = {1 \over e^x-1}+ {1 \over 2} 
\end{equation} 
Recall   the well-known functional identity
\begin{equation} \label{bfuncid} 
b(x_1)b(x_2) - b(x_1)b(x_2-x_1) +b(x_2) b(x_2-x_1)={1 \over 4} \ .
\end{equation}

\begin{defn} \label{defnsigmac} Let $n\geq 2$. 
Define elements $  \underline{\sigma}^c_{2n+1} \in \LL(\x_0,\x_1) / D^4 \LL(\x_0,\x_1)$
by
\begin{equation} \label{rhosimacdef} 
 \rho( \underline{\sigma}^c_{2n+1}) =   \underline{\xi}_{2n+1}  + \sum_{a+b = n}  {1 \over 2  b} \{   \underline{\xi}_{2a+1}, \{  \underline{\xi}_{2b+1},  \underline{\xi}_{-1}\}\}  \pmod{D^4}
 \end{equation}
where the sum is over $a,b\geq 1$.  Definition $(\ref{rhosimacdef})$ makes sense, since we shall prove in the next paragraph that the right-hand side    has no poles. 
Define
$$  \sigma_{2n+1}^{c} =          { (2n)!  \over B_{2n} }\underline{\sigma}^c_{2n+1} \qquad \hbox{ for } n \geq 2  $$
to be the canonical normalisations, 
and set $\sigma_{3}^{c} = [\x_0,[\x_0,\x_1]] + [\x_1,[\x_0,\x_1]]$.

\end{defn}

 Since  by Racinet's theorem, the space of solutions
to the double shuffle equations is closed under the Ihara bracket, this means that for $n\geq 2$
 the elements $\sigma^c_{2n+1}$ 
are solutions to the double shuffle equations in depths $\leq 3$.
\begin{rem} The  $\sigma_{2n+1}^c$ have the canonical  normalisation. For  $n\geq 1$,
$$(\sigma^c_{2n+1})^{(1)} = \xi^{(1)}_{2n+1}= x_1^{2n} \qquad \hbox{ and } \qquad (\sigma^c_{2n+1})^{(2)} =\xi^{(2)}_{2n+1}$$ is given by  $(\ref{xiindepth2})$. We have $(\sigma^c_3)^{(3)}=0$ and for $n\geq 2$,
\begin{equation} \label{sigmacexpl} 
 (\sigma^c_{2n+1})^{(3)} =   \xi_{2n+1}^{(3)} +     \sum_{a+b = n}  {B_{2a} B_{2b} \over B_{2n}} \binom{2n}{2a} {1 \over 24\,   b} \{  x_1^{2a}, \{ x_1^{2b}, x_1^{-2} \}\} \ .
 \end{equation}  
Due to the obvious   symmetry in $a$ and $b$, the previous expression can also be written in terms of lowest weight vectors for the action \cite{AA} of $\mathfrak{sl}_2$, namely:
 $$   {1 \over 2  b} \{  x_1^{2a}, \{ x_1^{2b}, x_1^{-2} \}\}  +  {1 \over 2  a} \{  x_1^{2b}, \{ x_1^{2a}, x_1^{-2} \}\}\ .$$
 On  the other hand, compare the odd part of the period polynomial \cite{KZ} for the Eisenstein series of weight $2n$, which is
proportional to:
$$\sum_{a+b=n, a, b\geq 1} \binom{2n}{2a} B_{2a} B_{2b} X^{2a-1} Y^{2b-1} \quad \in \quad \Q[X,Y]\ . $$
This is no accident, and follows from the computations in \S\ref{sectPolescancel} as well as \S\ref{sectfinal}.
\end{rem}

\subsection{Cancellation of poles} \label{sectPolescancel}
We show that $(\ref{sigmacexpl})$ has no poles. We need the following notation.
Given two even functions $f,g$ of one variable, define a new function  $$ (f\star g) (x_1,x_2)= f(x_1) g(x_2)  -f(x_2 \! -\! x_1) g(x_2) + f(x_2 \! -\! x_1) g(x_1) - f(x_2)  g(x_1)\ .$$
In the notation of \cite{AA} it is 
$f \star g  = f  \circb g  - g \, \studot f$. 

\begin{lem} For all $n, a,b \geq 1$, 
\begin{eqnarray} \label{residues}
  \Res_{x_3=0} \big( \xi^{(3)}_{2n+1} \big)  &  =  &  {1 \over 12}   x_1^{2n} \star  x_1^{-1}   \\ 
 \Res_{x_3=0} \big( \{x_1^{2a}, \{x_1^{2b}, x_1^{-2} \} \} \big)   & = &     2b\,  x_1^{2a} \star x_1^{2b-1} \ . \nonumber
\end{eqnarray} 
\end{lem} 
\begin{proof}
This is a straightforward computation and follows from the definitions. See also \cite{AA} for a generalisation of the second equation.
\end{proof}

\begin{prop} The elements $(\sigma^c_{2n+1})^{(3)}$ have no poles, for all $n \geq 2$.
\end{prop} 
\begin{proof} Since $\mathfrak{pls}$ (definition \ref{defnpls})  is a Lie algebra for the Ihara bracket and is contained in $Q'$, the element $ \{x_1^{2a}, \{x_1^{2b}, x_1^{-2} \}$ is in  $\mathfrak{pls}$  has at most simple poles along $x_1=0, x_2=x_1, x_3=x_2$ and $x_3=0$, whenever 
$a, b\geq -1$. We first check that the residue of $(\sigma^c_{2n+1})^{(3)}$  along  $x_3=0$ vanishes for $n\geq 2$. It is given via $(\ref{sigmacexpl})$ by 
$$  {B_{2n} \over (2n)!}   \Res_{x_3=0} \big( \xi^{(3)}_{2n+1} \big) +  {1 \over 12} \sum_{a+b=n} {B_{2a} \over (2a)!} {B_{2b} \over (2b)!} {1 \over 2b} \Res_{x_3=0} \big( \{x_1^{2a}, \{x_1^{2b}, x_1^{-2} \} \} \big)  \ .  $$
Pass to generating series and substitute  $(\ref{residues})$ into the previous expression to give
\begin{eqnarray} 
12  \sum_{n\geq 1 } \Res_{x_3=0} \,\underline{\sigma}_{2n+1}^{(3)}     & = &    (x b(x) -1 ) \star x^{-1} +    (x  b(x) -1 ) \star (b(x) -  x^{-1})   \nonumber \\
  & = & (x b(x)-1) \star b(x)    \ , \nonumber
   \end{eqnarray} 
   where $b(x)$ was defined in  $(\ref{bdef})$.
 To  compute this quantity, observe  that  $1 \star f =0$ and 
$$ (xf \star f) (x_1,x_2) = (x_1-x_2) \big( f(x_1)f(x_2) - f(x_1)f(x_2\! -\!x_1) +f(x_2) f(x_2\! -\!x_1)\big) \ . $$
for any even function $f$. Substituting for $f=b(x)$  and using $(\ref{bfuncid})$, we deduce that 
  $(x b(x)-1) \star b(x)= {1\over 4} (x_1-x_2)$. Now let $n\geq 2$. The above argument proves that the  $(\sigma^c_{2n+1})^{(3)}$  have no poles along  $x_3=0$. Now we use the fact that   $\sigma^c_{2n+1}$ satisfies the double shuffle equations modulo products in depths two and three.
   Since 
  $(\sigma^c_{2n+1})^{(i)}$ has no poles for $i=1,2$,   
   the stuffle equation $(\ref{stuffle})$ implies that 
  \begin{equation} \label{sigmaantipode}
(\sigma^c_{2n+1})^{(3)}(x_1,x_2,x_3) -  (\sigma^c_{2n+1})^{(3)}(x_3,x_2,x_1) \in \Q[x_1,x_2,x_3]  \  . 
  \end{equation}
  It follows that its residue at $x_1=0$ also vanishes. The shuffle equation is 
$$    (\sigma^c_{2n+1})^{(3)}(x_1, x_{12}, x_{123})+ (\sigma^c_{2n+1})^{(3)}(x_2, x_{12}, x_{123})+ (\sigma^c_{2n+1})^{(3)}(x_2, x_{23}, x_{123})  = 0  \ . $$
  By taking the residue of this expression at $x_2=0$, we deduce that  $(\sigma^c_{2n+1})^{(3)}$  has no pole along $x_2=x_1$.
  Finally by $(\ref{sigmaantipode})$, this implies that it has no pole along $x_2=x_3$ either.  \end{proof}
The last part of this argument is completely
  general and follows from the dihedral symmetry structure of the linearised double shuffle equations \cite{BrDepth}, \S6.3.

\begin{rem} \label{remz_3} The element $(\sigma^c_3)^{(3)}$ does have poles.
It is equal to $3 z_3$, where
$$z_3 ={4\over 3} +{ x_1 \over x_3-x_2}+{x_3\over x_1-x_2} + {x_3-x_2\over x_1}+{x_1-x_2\over x_3} $$
and corresponds to a lift to $ \Der^{\Theta} \, \LL(a,b)$ of the `arithmetic image' of the element  $i_1(\sigma_3) $ in $ (\Der^{\Theta} \, \LL(a,b))/\ue.$ 
The corresponding derivation  was written down  in \cite{Po}. Computing the stuffle equation $(\ref{linstuffle})$ gives
$$z_3(x_1,x_2,x_3) +z_3(x_2,x_1,x_3) +z_3(x_2,x_3,x_1)= 4$$
which is non-zero, and proves, by proposition \ref{propuedbshf}, that $z_3$ is non-geometric, i.e.,  not an element of $\ue$. 
\end{rem}

\subsection{Zeta elements in depth three}

\begin{thm}  \label{thmsigmasaremotivic} The elements $\sigma^c_{2n+1}$ are in the image of the map
$$i_0\quad : \quad   \g / D^4 \g \To \LL(\x_0,\x_1) / D^4 \LL(x_0,x_1) \ .$$
\end{thm} 
\begin{proof} The elements $\sigma_{2n+1}^c$ satisfy the double shuffle equations so lie in 
$D^1 /D^4 \dmr_0.$ The theorem follows immediately from the fact that 
$$i : \quad D^1 / D^4 \g \To D^1 /D^4 \dmr_0$$
is an isomorphism. This is equivalent to the statement that $i$ induces an isomorphism on  each depth-graded piece 
$$i: \quad \dg^{d} \cong \ls^d \qquad \hbox{ for } d\leq 3\ .$$
This is trivial for $d=1$, and follows from a computation of the dimensions of $\ls^{d}_n$ obtained by  Zagier  \cite{Za} for $d=2$, and  by Goncharov (\cite{G2}, theorem 1.5) for $d=3$.
  \end{proof}

\begin{rem} It follows  from the depth-parity theorem that the elements $\sigma^c_{2n+1}$ are uniquely determined in depth 4 also (but not in depth 5). 
A closed formula for these elements can   be  deduced from remark \ref{rem extend}. 
 \end{rem}

\section{Zeta elements in depth 3 via geometric derivations} \label{sectDrinGeom}

Recall the notations from \S\ref{sectElliptic}.

\begin{thm}  \label{thmsigmasepsilons} For all $n\geq 2$, we have an explicit expansion
$$i_1(\underline{\sigma}^c_{2n+1}) \equiv \underline{\varepsilon}^{\vee}_{2n+2} + \sum_{a+b=n}  
{1 \over 2b} \big[ \underline{\varepsilon}^{\vee}_{2a+2}, \big[ \underline{\varepsilon}_{2b+2}^{\vee},\underline{\varepsilon}_0^{\vee} \big] \big]
\pmod{B^4}$$
For $n=1$, we have $i_1(\sigma^c_3) = \varepsilon_4^{\vee} + z_3 \pmod{B^4}$, where $z_3$ is defined below.
\end{thm}
In terms of the standard normalisations, the previous equation is equivalent to
$$i_1(\sigma^c_{2n+1} ) \equiv \varepsilon_{2n+2}^{\vee} + \sum_{a+b=n} {B_{2a} B_{2b} \over B_{2n}} \binom{2n}{2a} {1 \over 24b } \big[ \varepsilon_{2a+2}^{\vee}, \big[ \varepsilon_{2b+2}^{\vee}, \varepsilon_0^{\vee}\big]\big]  \pmod{B^4}$$
One can also write the right-hand side  symmetrically using  elements
$$\mathrm{lw}_{a,b} =  {1 \over 2b} \big[ \underline{\varepsilon}^{\vee}_{2a+2}, \big[  \underline{\varepsilon}^{\vee}_{2b+2} , \underline{\varepsilon}^{\vee}_0\big] \big]+  {1 \over 2a} \big[ \underline{\varepsilon}^{\vee}_{2b+2}, \big[ \underline{\varepsilon}^{\vee}_{2a+2} , \underline{\varepsilon}^{\vee}_0\big] 
\big]\ . $$
These are lowest weight vectors for the action of $\mathfrak{sl}_2$, i.e.,  $\varepsilon_0 \mathrm{lw}_{a,b}=0$, 
where $\varepsilon_0$ is the derivation on $\LL(a,b)$ such that $\varepsilon_0(a)=0$ and $\varepsilon(b)=a$.
This yields a direct comparison with the period polynomials of Eisenstein series.

The strategy for the proof is as follows. The elements $\varepsilon_{2n+2}^{\vee}$ do not preserve the image of $\Der^1 \, \LL(\x_0,\x_1)^{\wedge}$ under $(\ref{phidefn})$ and do not descend
to derivations on $\LL(\x_0,\x_1)^{\wedge}$. However, if we pass to  commutative power series  representations via $\S\ref{algtopoly}$ and enlarge them  by introducing poles, then the  
 elements $\varepsilon_{2n+2}^{\vee}$, considered modulo $B^4$, descend to the elements $\xi_{2n+1}$ defined in the previous section.
 Theorem $\ref{thmsigmasepsilons}$ is then equivalent to theorem $\ref{thmsigmasaremotivic}$ via definition \ref{defnsigmac}.

 The proof of theorem $\ref{thmsigmasepsilons}$ given here is from the `bottom up': i.e. by lifting the analogous theorem on $\Pro^1 \backslash \{0,1,\infty\}$. 
 A different way to prove theorem $\ref{thmsigmasaremotivic}$ from the `top down' via $\mathcal{M}_{1,1}$, is sketched in the final section of the paper.

\subsection{Hain homomorphism in low depth}
Throughout this paragraph, we  apply the method of commutative power series \S$\ref{algtopoly}$, \S$\ref{sectLabalgtopoly}$  to  both $\LL(\x_0,\x_1)$ and $\LL(a,b)$.

We wish to consider the depth $r$ components of the Hain morphism \S\ref{sectHainmorphism}:
 $$
 \phi^r: \gr^{\bullet}_D \, T(\x_0,\x_1)  \To  \gr^{\bullet+r}_B \, T(a,b) \ .
$$
Translating into rational functions \S\ref{sectCommSeries}, and using the fact that  $\rho:  \gr^{\bullet}_D T(\x_0,\x_1) \cong P$
is an isomorphism, we obtain a commutative diagram
\begin{equation}
\begin{array}{ccc} \label{phionratfunc}
    \gr^{\bullet}_D \, T(\x_0,\x_1) & \overset{\phi^r}{\To} &\gr^{\bullet+r}_B\,  T(a,b)  \\
  \downarrow_{\rho} &   & \downarrow_{\ell}   \\
    P &   \overset{\phi_r}{\To}   & Q   \ .
\end{array}
\end{equation}
The map along the top is denoted by a superscript, the one along the bottom  by  a subscript.   The map $\phi^0$ is simply the associated graded
of  $(\ref{phidefn})$:
\begin{eqnarray} \label{gr0phi}
\phi^0 = \gr \, \phi:  \begin{cases}   \x_0 \mapsto a  \\
  \x_1 \mapsto [a,b]  \end{cases}
  \end{eqnarray} 
and via $T(\x_0,\x_1) = \gr_D T(\x_0, \x_1)$ and $T(a,b) = \gr_B T(a, b)$, we have $\phi = \sum_{r\geq 0} \phi^r$.
The idea of the following discussion is to  factorise the map $\phi$ as a composition
$$\LL(\x_0, \x_1)  \overset{\phi^0}{\To} \LL(a, [a,b])  \overset{p}{\To} \LL(a,b)^{\wedge}$$
where $p:  \LL(a,[a,b]) \rightarrow \LL(a,b)^{\wedge}$ is a continuous map which satisfies $p([a,b])=0$,  and to express $p$, in low degrees, via its rational function representation.

\begin{lem}  \label{Eqphi0} The map $\phi_0:P \rightarrow Q$ is the  inclusion $P \subset Q$.
\end{lem} 
\begin{proof}
We show that  $\phi_0: \Q[y_0, y_1,\ldots, y_r]   \rightarrow   \Q[y_0, y_1,\ldots, y_r] $
 is   multiplication by  the element  $\ell_r$ of $(\ref{ellrdefn})$.
  To see this,  $(\ref{gr0phi})$  is the map $\x_0 \rightarrow a, \x_1 \rightarrow b$, 
followed by the composition of $r$ maps, where the $k^{\mathrm{th}}$ map, for $1\leq k \leq r$, replaces the $k^{\mathrm{th}}$ occurrence of the letter $b$  in $a^{i_0} b a^{i_1} b \ldots a^{i_{r-1}} b a^{i_r}$ with $ab-ba$. On commutative power series $(\ref{rhosecond})$, this is
multiplication by $y_{k-1}-y_k$.  \end{proof}

 We next determine $\phi_r$ for $r=1,2$. In degree $r=1$,   it follows from the definition $(\ref{phidefn})$ of $\phi(\x_0) = a + {1\over 2} [a,b] + \ldots$ that it is a composition
 of $\phi^0$ $(\ref{gr0phi})$,  whose image consists of words in $a, [a,b]$, followed by the derivation in  $D^{\Theta}$ (\S\ref{sectLabalgtopoly}) which sends $a \mapsto  {1 \over 2} [a,b]$ and $[a,b]$ to zero.  Note that the latter does not extend to an element of $\Der^{\Theta} \, \LL(a,b)$. It is nonetheless represented, via proposition \ref{propAction2}, by 
$$-s^{(1)}  = {1 \over 2 } \ell'( [a,b])   = {1 \over 2 (y_0-y_1)}  \in Q'\ ,$$  whose reduced representation is minus $(\ref{sdef})$.
Therefore  if $f\in P$, we have 
\begin{equation} \label{Eqphi1} 
\phi_1(f) =  -s^{(1)} \eact  \phi_0(f)\  . 
\end{equation}
    Similarly, in degree $r=2$, we have  for $f\in P$, 
\begin{equation} \label{Eqphi2} 
\phi_2(f)  =  {1\over 2} s^{(1)} \eact (s^{(1)} \eact \phi_0(f)) - s^{(2)} \eact \phi_0(f)  \ ,  
\end{equation}
where 
$$-s^{(2)}  = {1 \over 12} \ell' ( [b,[b,a]])  =   {1 \over 12}   {y_0 -2 y_1 +y_2 \over  (y_0-y_1)(y_1-y_2)(y_0-y_2) }      \in Q'\ .$$ 
This holds from the definition of $\phi$,  since $\phi^2- {1\over 2} \phi^1\phi^1$ is the composition of $\phi^0$
followed by the derivation $\LL(a,[a,b])\rightarrow \LL(a,b)$ which sends $a \mapsto  {1 \over 12} [b,[b,a]]$ and $[a,b]$ to zero. By proposition $\ref{propAction2}$ the latter  corresponds
to the action of $s^{(2)}$. 

\begin{rem} Because elements of $Q'$ act trivially on $\ell([a,b])=-s^{(1)} \eact y_0$, we have
\begin{equation} \label{s1comm}
 s^{(1)} \eact ( f \eact y_0) = \{s^{(1)},f\} \eact y_0  \quad  , \quad \hbox{for all  } f\in Q' \ .
 \end{equation} 
This can also be read off corollary $\ref{cordx1}$ upon writing 
$  \phi^1 = {1 \over 2} \phi^0 \circ  \partial_{\x_1}$.\end{rem}

\subsection{Proof of theorem \ref{thmsigmasepsilons}}
Recall that an element $\sigma \in \Der^1 \LL(\x_0,\x_1)$ lifts to $\widetilde{\sigma} \in \Der^1 \LL(a,b)$ if and only if the following equation holds  in $\LL(a,b)$:
$$\widetilde{\sigma} \phi(\x_0) = \phi \sigma(\x_0)$$
Finding an element $\sigma \in \Der^1 \LL(\x_0,\x_1)$ whose lift is $\varepsilon_{2n+2}^{\vee}$, is equivalent via $(\ref{phidefn})$,  modulo terms of $B$-degree $\geq 4$,  to the following equation
$$ \varepsilon_{2n+2}^{\vee} \big( a + {1 \over 2} [a,b] + {1 \over 12} [b,[b,a]] \big) \equiv \phi(  \sigma(\x_0)) \pmod{B^4}\ .$$
It is easy to verify that it has no solution $\sigma \in \Der^1\, \LL(\x_0,\x_1)$. Note that
$ \varepsilon_{2n+2}^{\vee} ([a,b]) =0$ so the middle term on the left-hand side can be dropped.
 We can pass to rational function representations via propositions $\ref{propAction1}$ and $\ref{propAction2}$,
 and view the previous equation in $Q$. Since $\ell'( \varepsilon_{2n+2}^{\vee} ) = (y_1-y_0)^{2n}$, and $\rho(a)=y_0$, $\rho(\x_0)=y_0$,   it is equivalent to 
\begin{equation} \label{chicondition} 
 (y_1-y_0)^{2n} \eact \Big(\big(1 -  s^{(2)} \big) \eact y_0 \Big) \equiv \phi(  \rho'(\sigma) \odot y_0)  \pmod{B^4}
 \end{equation} 
 It has no solutions $\rho'(\sigma)\in  P'$. Now observe that 
 $$\phi_0 ( \rho'(\sigma) \odot y_0) = \phi_0 (\rho'(\sigma)) \eact y_0$$
 since the formulae for $\odot$ and $\eact$ (propositions \ref{propAction1} and \ref{propAction2}) are formally identical and $\phi_0$ is the identity.
 Let us write $\chi$ instead of $ \phi_0 (\rho'(\sigma))$ and try to solve $(\ref{chicondition})$ for $\chi \in Q'$.
 By  $(\ref{Eqphi1})$, $(\ref{Eqphi2})$,  the right-hand side of $(\ref{chicondition})$ is equal,  after expanding $\phi\equiv \phi^0 + \phi^1 + \phi^2 \mod B^4$ and  applying  $(\ref{s1comm})$, to 
$$  \chi  \eact y_0   -   \{ s^{(1)} , \chi \} \eact y_0    + {1 \over 2} \{ s^{(1)},\{ s^{(1)} , \chi \}\} \eact y_0 
  - s^{(2)} \eact \big(  \chi \eact y_0 \big) \pmod{B^4} \ .$$
   Using the fact  (proposition $\ref{propAction2}$) that the Lie bracket $\{ \  , \  \}$ is the antisymmetrization of $\eact$, 
   and that the action of $Q'$ on $y_0 \in Q$ is faithful, we deduce that the 
    components of equation $(\ref{chicondition})$ in  degrees $1,2,3$ are the equations:
\begin{eqnarray} 
 (y_1-y_0)^{2n}  &= &    \chi^{(1)} \nonumber \\
 0 & = &     \chi^{(2)}-  \{ s^{(1)} ,\chi^{(1)}\}    \nonumber \\
-(y_1-y_0)^{2n} \eact s^{(2)}  & = & \chi^{(3)} - \{s^{(1)}, \chi^{(2)}\}+ {1 \over 2}  \{ s^{(1)},\{ s^{(1)} ,\chi^{(1)} \}\} -   s^{(2)} \eact \chi^{(1)}
  \ .
 \nonumber 
 \end{eqnarray}
 These three equations are 
 equivalent to the definition of the elements $\chi_{2n+1}$ after passing to reduced versions $(y_0,y_1,y_2) \mapsto (0,x_1,x_2)$
 and using  definition $(\ref{eactdefn})$.

\section{Explicit rational associator  in depths $\leq 3$}
In section \S$\ref{sectDrin3}$ we wrote down explicit solutions to the double shuffle equations modulo products,  in  odd weights and depths $\leq 3$. 
The goal of this paragraph is  to discuss solutions to the full double shuffle equations with even weights in the same range.
 
\subsection{Double shuffle equations} \label{secttau}
The full double shuffle equations in depth two are given by the pair of equations:
\begin{eqnarray} \label{Fulldepth2}Ê
f^{(2)}(x_1,x_1+x_2) + f^{(2)}(x_2,x_1+x_2)&  =  &f^{(1)}(x_1)f^{(1)}(x_2) \nonumber  \\
f_{\star}^{(2)}(x_1,x_2) + f_{\star}^{(2)}(x_2,x_1) & = &  { f_{\star}^{(1)}(x_1) -f_{\star}^{(1)}(x_2) \over x_2-x_1} +  f_{\star}^{(1)}(x_1)f_{\star}^{(1)}(x_2) 
\end{eqnarray}
where $f^{(1)}, f_{\star}^{(1)} \in \Q[[x_1]]$ and $f^{(2)} , f_{\star}^{(2)} \in \Q[[x_1,x_2]]$ are formal power series in commuting variables. 
A  power series $f$ without a subscript will  denote  its shuffle-regularised version; a subscript $\star$ will denote 
its stuffle-regularised version. They differ by a factor which is well-understood \cite{Ra}. Our normalisations will be such that 
$$f^{(1)}_{\star}=f^{(1)}  \quad  \hbox{ and } \quad  f^{(2)}_{\star}= f^{(2)} + {1\over 48} $$
One can easily convince oneself that the second equation of $(\ref{Fulldepth2})$ is the direct translation of  the stuffle product formula
$$\zeta(m,n) + \zeta(n,m) + \zeta(m+n) = \zeta(m)\zeta(n)\ .$$
Note that, in contrast to the double shuffle equations modulo products,  the right-hand term in the previous equation means that we must consider all weights simultaneously.  
The  shuffle equation in depth three takes the form 
$$
f^{(3)} (x_1,x_{12},x_{123}) + f^{(3)}(x_2,x_{12}, x_{123}) +   f^{(3)}(x_2,x_{23}, x_{123}) = f^{(1)}(x_1) f^{(2)}(x_2,x_{23}) 
$$
and the stuffle equation takes the form
\begin{multline}
f_{\star}^{(3)}(x_1, x_2, x_3)+f_{\star}^{(3)}(x_2, x_1, x_3)+f_{\star}^{(3)}(x_2, x_3, x_1) \quad   =   \quad   \nonumber \\
   { f_{\star}^{(2)}(x_2, x_1)-f_{\star}^{(2)}(x_2,x_3)  \over  x_3-x_1}       +  { f_{\star}^{(2)}(x_1, x_3)-f_{\star}^{(2)}(x_2,x_3) \over  x_2-x_1}   
   + f_{\star}^{(1)} (x_1) f_{\star}^{(2)} (x_2, x_3)  \ .  \nonumber 
\end{multline}
where in this case the comparison between the two regularisations is given by 
$$f^{(3)}(x_1,x_2,x_3) = f_{\star}^{(3)}(x_1,x_2,x_3) +{1 \over 96}( b(x_1) - {1 \over x_1}) $$

The general principle \cite{AA} of constructing solutions to these equations with poles and correcting with counter terms also holds in this situation. The full double shuffle equations are inhomogeneous in two different ways:  there are several linear  terms of lower depths and a single term consisting of  products of elements of lower depth. The strategy is to construct solutions $\gamma$ to the equations in which lower depth terms are omitted, but with all product terms retained, and to use the element
$(\ref{sdef})$   to convert these solutions into polar solutions to the full equations. The polar parts are then subtracted using counterterms 
involving the elements $\xi_{2n+1}$ constructed before.
\subsection{Polar solutions}
Recall that  $b_1(x)=b(x) $  $(\ref{bdef})$ is a   generating series for Bernoulli numbers whose Laurent series is 
$x^{-1} + O(x)$.  Thinking of $b_1(x)$ as a deformation  of the rational function 
 $x^{-1}$, leads us  to introduce,  following $(\ref{sdef})$, the function
$$b_2(x_1,x_2)= {1 \over 3} \big( b_1(x_1)  b_1(x_2) +  b_1(x_2) b_1(x_1-x_2)   \big) \ .  $$
With these definitions, set
\begin{eqnarray}
2\gamma^{(1)} &=  & -b_1 \nonumber \\
4 \gamma^{(2)} &=  &  -b_2 + \textstyle{1\over 2} b_1 \circb b_1 \nonumber \\
 8 \gamma^{(3)} &=  &  -b_2 \circb b_1  + \textstyle{ 1\over 6} b_1 \circb ( b_1 \circb b_1)\nonumber 
\end{eqnarray}
The element $\gamma^{(2)}$, for example, solves the semi-homogeneous equations
\begin{eqnarray} \label{semihom}Ê
\gamma^{(2)}(x_1,x_1+x_2) + \gamma^{(2)}(x_2,x_1+x_2)  &= & \gamma^{(1)}(x_1)\gamma^{(1)}(x_2) \\
\gamma_{\stu}^{(2)}(x_1,x_2) + \gamma_{\stu}^{(2)}(x_2,x_1)  & =  &  \gamma^{(1)}(x_1)\gamma^{(1)}(x_2) \nonumber 
\end{eqnarray}
where $\gamma_{\stu}^{(2)} = \gamma^{(2)} + {1 \over 48}$, and  $\gamma^{(3)}$ satisfies the equations
\begin{eqnarray} \label{semihom}Ê
\gamma^{(3)}(x_1,x_{12}, x_{123}) + \gamma^{(3)}(x_2,x_{12}, x_{123}) + \gamma^{(3)}(x_2,x_{23}, x_{123})  &= & \gamma^{(1)}(x_1)\gamma^{(2)}(x_2,x_{23}) \nonumber \\
\gamma_{\stu}^{(3)}(x_1,x_2,x_3) + \gamma_{\stu}^{(3)}(x_2,x_1,x_3)  + \gamma_{\stu}^{(3)}(x_2,x_3,x_1)  & =  &  \gamma^{(1)}(x_1)\gamma_{\stu}^{(2)}(x_2,x_3) \nonumber 
\end{eqnarray}
where  $\gamma_{\stu}^{(3)}(x_1,x_2,x_3) = \gamma^{(3)}(x_1,x_2,x_3) + {1 \over 48} \gamma^{(1)}(x_1)$. 
This follows easily from $(\ref{bfuncid})$.

New  series $\Theta$ are now defined   by twisting   on the left  by  the elements $(\ref{sdef})$:
\begin{eqnarray} 
\Theta^{(1)} & = &  \gamma^{(1)}   \nonumber  \\
\Theta^{(2)} & = &  \gamma^{(2)} +  s^{(1)} \circb \gamma^{(1)}   \nonumber  \\
\Theta^{(3)} & = &    \gamma^{(3)}  + s^{(1)} \circb \gamma^{(2)}  +s^{(2)} \circb \gamma^{(1)}  + \textstyle{1 \over 2} s^{(1)} \circb (  s^{(1)} \circb \gamma^{(1)} )  \nonumber 
\end{eqnarray}
They  have poles in $x_i$. More precisely,  the element   $d_r \Theta^{(r)} $ can be viewed as a formal power series
in $\Q[[x_1,\ldots, x_r]]$, where $d_r = x_1\ldots x_r \prod_{i<j} (x_i-x_j)$,  for $1\leq r\leq 3$.

For $1 \leq r \leq 3$, we can write 
$$\Theta^{(r)} = p_r + \Phi^{(r)}\ ,$$  where $p_r$ is a  homogeneous rational function in $x_1,\ldots, x_r$ of degree $-r$,
and $\Phi^{(r)}$ is a power series in homogeneous rational functions of degrees $> 1-r$.  With these definitions,  one verifies that the truncated elements
$\Phi^{(r)}$ are polar solutions to the full  double shuffle equations \S\ref{secttau} 
 with
$$\Phi_*^{(2)} = \Phi^{(2)}+ {1 \over 48} \qquad \hbox{ and } \qquad \Phi_*^{(3)} = \Phi^{(3)} + {1 \over 48} \Phi^{(1)}(x_1) \ .$$
It remains to remove the poles from the $\Phi^{(r)}$ to obtain bona fide  polynomial solutions to the double shuffle equations  with no polar terms.

\subsection{Subtraction of counterterms}
Let us define a formal power series by 
\begin{equation}
C = \sum_{n \geq 1} {1 \over 2n} \{ \underline{\xi}_{-1}, \underline{\xi}_{2n+1}\}
\end{equation} 
where the elements $\underline{\xi}_{2n+1}$ were defined in \S\ref{sectChis}. Its definition was only given in depths $1,2,3$. 
Using this element to provide counter terms, we can finally  write down a canonical element $\tau$ in depths $1,2,3$ as follows:
\begin{eqnarray} 
\tau^{(1)} & = &   \Phi^{(1)}    \nonumber  \\
\tau^{(2)} & = &  \Phi^{(2)} + C^{(2)}  \nonumber  \\
\tau^{(3)} & = &    \Phi^{(3)} +C^{(2)} \circb \Phi^{(1)} +C^{(3)}\nonumber 
\end{eqnarray}
A straightforward residue computation along the lines of \S\ref{sectPolescancel} suffices to show  that the elements $\tau^{(i)}$, where  $i=1,2,3$ have no poles, and therefore lie in $\Q[[x_1,\ldots, x_i]]$. We omit the  details.
Note that by the depth-parity theorem, the element $\tau^{(3)}$ is uniquely determined from $\tau^{(2)}$. By a version of \S\ref{algtopoly}, the coefficients of $\tau^{(i)}$ correspond
to words in $\x_0,\x_1$, and taking the limit defines a unique  element 
$$\tau \in \Q\langle \langle \x_0, \x_1 \rangle \rangle/D^4 \Q\langle \langle e_0, e_1 \rangle \rangle\ .$$
\begin{thm} The element $\tau$ is an explicit  (shuffle-regularized)   solution to the full double shuffle equations in depths $\leq 3$. 
\end{thm}
A similar construction holds in depth four, but there is \emph{a priori} no canonical way to cancel the poles:
  one must subtract counter-terms
consisting of quadruple brackets in the $\xi_{2n+1}$'s, which involves some choices because of  quadratic relations amongst them (see \S\ref{sectQuadrel}). It is an interesting question to ask if the element $\tau$ defined above 
can in fact  be  extended to an explicit  associator in  higher depths.

Since the solutions to the full double shuffle equations is a torsor under the left action of the prounipotent algebraic group $\mathrm{DMR}_0$ whose Lie algebra is $\dmr_0$, 
we can twist the elements $\tau^{(i)}$ on the left with our canonical elements $\exp_{\!\circb} \! \sigma_{2n+1}^c$ to obtain all other rational solutions 
to the double shuffle equations in depths $\leq 3$.
\begin{cor}  Every rational solution $s$ to the full double shuffle equations  in depths $\leq 3$  can be written explicitly  in the form
$$ s  \equiv \exp_{\!\circb}\!(g)\circb \tau \pmod{D^4}$$
where $g \in (\g/D^4\g)^{\wedge}$ is  an (infinite) linear combination of commutators in the canonical elements $\sigma_{2n+1}^c$
 of length $\leq 3$.\end{cor} 
Note that the element $g$ in the corollary is not unique because of quadratic relations \S\ref{sectQuadrel} in $ \g /D^4 \g$.

\subsection{Remarks} The elements $\tau^{(i)}$ for $i\leq 3$ define a homomorphism from motivic multiple zeta values  in depths $\leq 3$ and even weight to rational numbers, given by 
\begin{equation} \label{taumap} 
\tau^{(r)}  \zetam(n_1, \ldots, n_r) =  \hbox{ coeff. of } x_1^{n_1-1} \ldots  x_r^{n_r-1}  \hbox{ in } \tau^{(r)}\ .
\end{equation} 
They respect all the relations between motivic multiple zeta values and satisfy
$$\tau^{(1)}  \zetam(2n) = {\zeta(2n)/ ( 2\pi i)^{2n}} \quad \in \Q \ . $$
Likewise, the canonical elements $\sigma_{2n+1}^c \in \g /D^4 \g$ define a map from motivic multiple zeta values  in depth $\leq 4$ and odd weight  to rational numbers given by 
\begin{equation} \label{sigmamap} 
\sigma_{2n+1}^{(r)}  \zetam(n_1, \ldots, n_r) =  \hbox{ coeff. of } x_1^{n_1-1} \ldots  x_r^{n_r-1}  \hbox{ in } \sigma^{(r)}
\end{equation} 
where $2n+1= n_1+\ldots + n_r$. The maps $(\ref{sigmamap})$  annihilate products,  respect all relations between motivic multiple zeta values (modulo products) and satisfy
$$\sigma_{2n+1}^{(1)}  \zetam(2n+1) = 1 \ . $$
In \cite{BrDec}, a method was described to decompose any motivic multiple zeta value (and hence, by taking the period, any actual multiple zeta value) into a chosen basis of motivic multiple zeta values using the motivic coaction. The method is not an algorithm because it requires a transcendental computation at each step using the period map. However, the maps $(\ref{taumap})$ and $(\ref{sigmamap})$ can be used as a substitute for the period map. Thus we obtain as a corollary an exact algorithm to decompose any multiple zeta value of depth $\leq 3$ (and depth $\leq 4$ in the case of odd weight) into a chosen basis of multiple zeta values of the same or smaller depth.

\section{Cuspidal elements and the Broadhurst-Kreimer conjecture} \label{sectCuspelements}
We can recast the version of the Broadhurst-Kreimer conjecture stated in \cite{BrDepth}  using the  $\sigma^c_{2n+1}$,  first in   $\gr_D\, \Der^1\, \LL(\x_0,x_1)$ and then in the elliptic setting in $\gr_B \,\Der^{\Theta} \,\LL(a,b)$.

We seek a conjectural presentation for $\dg$. The first set of  obvious generators are the images of the  zeta elements $\sigma_{2n+1} \in D^1 \g$ in the associated graded $\dg^{\bullet}=\gr^{\bullet}_D \g$:
\begin{equation} 
\overline{\sigma}_{2n+1} \in \dg^1_{2n+1} \qquad \hbox{ for all } n\geq 1\ .
\end{equation} 
They are well-defined (independent of the choice of $\sigma_{2n+1}$). They satisfy quadratic relations 
which can be  described in terms of period polynomials.

\subsection{Reminders on period polynomials}
Let $n\geq 0$ and let $V_n = \bigoplus_{i+j=n} \Q x_1^i x_2^j$ denote the vector space of homogeneous polynomials of degree $n$. It is 
equipped with the right action of $\Gamma= \SL_2(\Z)$  given by the formula
$$P(x_1,x_2)|_{\gamma} = P( ax_1+bx_2, cx_1+dx_2) \quad \hbox{ if }\quad \gamma =   \left( \begin{smallmatrix} a&b\\ c&d \end{smallmatrix} \right) \in \Gamma \ , P \in V_n \ . $$
Let $V'_n \subset V_n $ denote the subspace of polynomials which vanish at $x_1=0$ and $x_2=0$. It is naturally isomorphic to  the  vector space quotient $V_n / (\Q x_1^n \oplus \Q x_2^n)$.
\begin{defn} \label{defperiodpoly}  Let $n\geq 1$ and let $S_{2n}\subset V'_{2n}$ denote the vector space of homogeneous polynomials $P(x_1,x_2)$  of degree  $2n$ satisfying  $P(x_1,0)=P(0,x_2)=0$  and 
$$P(x_1,x_2) + P(x_2,x_1)=  0  \quad \ , \  \quad P(x_1, x_2) + P(x_1-x_2,x_1) + P (-x_2, x_1-x_2) =0  \ . $$
The subspace $S_{2n}^+\subset S_{2n}$ consisting of polynomials which are of even degree in both $x_1$ and $x_2$ is called the space 
of even (cuspidal) period polynomials.
\end{defn}

\begin{rem} Denote the standard  elements  $S =   \left( \begin{smallmatrix} 0&-1\\ 1&0 \end{smallmatrix} \right)$ and $T=  \left( \begin{smallmatrix} 1&1\\ 0&1 \end{smallmatrix} \right)$ in $\Gamma$. Consider the following linear map from right $\Gamma$ group cochains (\cite{MMV}, \S2.3) to polynomials 
\begin{equation} \label{Z1topolys}  f\mapsto \pi(f(S))  : Z_{\cusp}^1(\Gamma; V_{2n})  \To  V'_{2n}  \end{equation}
where $Z_{\cusp}^1(\Gamma;V_n)\subset Z^1(\Gamma;V_n)$ is the subgroup of cochains $f$ such that $f(T)=0$,  and $\pi: V_n \rightarrow V'_n$
is the projection. It is well-known that this induces an isomorphism
$$ H^1_{\cusp}(\Gamma; V_{2n} ) \overset{\sim}{\To} S_{2n}$$
where  $H^1_{\cusp}(\Gamma;V_{2n})=\ker (H^1(\Gamma;V_{2n}) \rightarrow H^1(\Gamma_{\infty}; V_{2n}))$, and $\Gamma_{\infty} \leq \Gamma$ is the subgroup generated by $-1, T$.  This in turn induces an isomorphism 
\begin{equation}\label{H1toS2n} 
 H^1_{\cusp}(\Gamma; V_{2n} )^+ \overset{\sim}{\To} S_{2n}^+
 \end{equation}
where the $+$ on the left-hand factor denotes invariants with respect to the action of the real Frobenius involution (\cite{MMV} \S5.4, \S7.4).
The Eichler-Shimura theorem states in particular that the integration maps gives an  isomorphism:
$$S_{2n}(\Gamma) \overset{\sim}{\To}   H^1_{\cusp}(\Gamma; V_{2n+2} )^+ \otimes \R$$
where $S_{2n}$ denotes the space of cuspidal modular forms of weight $2n$. 
\end{rem}

\subsection{Quadratic relations} \label{sectQuadrel}
Define a   vector space $K$ 
$$K   =   \ker ( \{ \ , \}: \dg^1 \wedge \dg^1 \rightarrow \dg^2)$$
to be the kernel of the Ihara bracket.  It is weight-graded in even degrees $K= \bigoplus_n K_{2n}$. Since  
$\g \subset D^1 \g$ is generated in depth 1, $D^1 \wedge D^1 \rightarrow D^2$ is surjective, and hence 
 $\dg^1 \wedge \dg^1 \rightarrow \dg^2$ is surjective, i.e.,
$$0 \To K \To \dg^1 \wedge \dg^1 \overset{\{ \ , \ \}}{\To} \dg^2 \To 0  $$
is an exact sequence.  Now embed $\g$  in $\LL(\x_0,\x_1)$ via $(\ref{gfraktoL})$, and therefore
$\dg=\gr_{D} \g$ is also embedded in $\LL(\x_0,\x_1)$ since the latter is graded for the $D$-degree.
By passing to reduced polynomial representations \S\ref{algtopoly}, we have a canonical isomorphism
\begin{eqnarray}
\dg^1  & = &  x_1^2 \Q[x_1^2] \nonumber \\
\overline{\sigma}_{2n+1} & \mapsto & x_1^{2n} \qquad   \qquad \hbox{ for } n\geq 1 .\nonumber
\end{eqnarray}  
We can thus identify $\dg^1 \otimes \dg^1  =  x_1^2 \Q[x_1^2] \otimes x_1^2 \Q[x_1^2] \cong x_1^2x^2_2 \Q[x_1^2,x_2^2]$, 
and hence   view elements of $\dg^1 \wedge \dg^1  \subset \dg^1 \otimes \dg^1 $ as antisymmetric polynomials in $x^2_1,x^2_2$.

\begin{lem}   \label{lemquadrel} The polynomial representation gives an isomorphism 
\begin{equation} \label{isomKtoS}
   K_{2n} \overset{\sim}{\rightarrow} S_{2n}^+ \ . 
   \end{equation}
\end{lem} 
\begin{proof}
This is immediate from the formula for $\circb$ given in \S\ref{sectLinIhara} (example \ref{exampleofIharaactionlength2}) 
$$\{x_1^{2a}, x_1^{2b}\} = P(x_1,x_2)  + P(x_2-x_1,x_1) + P(-x_2,x_1-x_2)  $$
where $P(x_1,x_2) = x_1^{2a}x_2^{2b}- x_2^{2a}x_1^{2b}$ and $a,b\in \N$.
\end{proof}
These quadratic relations appear in several contexts:
\begin{cor} \label{corquad}
The elements $\varepsilon_{2n+2}^{\vee}$, $\xi_{2n+1}$, $\overline{\sigma}_{2n+1}$ and $x_1^{2n}$  all satisfy the same quadratic relations in $K_{2n}$.
\end{cor} 
\begin{proof}
The polynomial representations of $\varepsilon_{2n+2}^{\vee}$ and $\overline{\sigma}_{2n+1}$ via propositions $\ref{propAction1}$ and $ \ref{propAction2}$ are both $x_1^{2n}$, and the Lie brackets correspond to the Ihara bracket $\{ ,\}$. Therefore they satisfy the identical quadratic relations. 
For the elements $\xi_{2n+1}$, this follows from their definition, because they are obtained from the $x_1^{2n}$
via the Ihara bracket, or  from the computations of \S\ref{sectDrinGeom} relating them to the $\varepsilon_{2n+2}^{\vee}$.
\end{proof} 
The existence of such quadratic relations was first observed by Ihara-Takao and has been reproved in  many ways since.
The smallest example of a period polynomial is the element $x_1^2x_2^2(x_1^2-x_2^2)^3  = x_1^8 x_2^2 -3 x_1^6 x_2^4+3 x_1^4 x_2^6-x_1^2x_2^8$.
It corresponds to the relations
\begin{eqnarray} \label{IharaRel} 
  { [}\overline{\sigma}_3, \overline{\sigma}_9] - 3 [ \overline{\sigma}_5, \overline{\sigma}_7]  =  0 \qquad \ , \    
 \qquad   \{ x_1^2, x_1^8\} - 3 \{ x_1^4, x_1^6\} =   0   \ .\nonumber   \\
{[} \varepsilon^{\vee}_{4}, \varepsilon^{\vee}_{10} ] - 3 [ \varepsilon^{\vee}_{6} , \varepsilon^{\vee}_8] =   0 \quad \,\, \ , \  \quad     
 \qquad     \{\xi_3, \xi_9\} - 3 \{ \xi_5, \xi_7\}  =  0  \ .\nonumber   
\end{eqnarray}

\subsection{Cuspidal generators in depth 4} \label{sectExotic} As explained in \cite{BrDepth}, the depth filtration on $\g$ gives rise
to a spectral sequence and in particular a differential 
$$d : H_2(\dg) \To H_1(\dg)\  . $$
Since $H_2(\dg) = \ker(\wedge^2 \dg \rightarrow \dg)/ \wedge^3 \dg$, there is a natural map 
  $K\rightarrow H_2(\dg)$. It is in fact injective since the image  of $\wedge^3 \dg $ is in depth $\geq 3$. Composing with  this map gives
   a  linear map $d: K \rightarrow (\dg^4)^{ab}$ as we explain presently, and the canonical zeta elements
defined in \S\ref{sectDrin3}  give a means to compute it explicitly.  To see this, the elements $\sigma^c_{2n+1}$ can be interpreted as  a linear map  
\begin{eqnarray} 
\sigma^c: \dg^1  &\To &  D^1 \g / D^4 \g  \nonumber \\ 
\overline{\sigma}_{2n+1}  & \mapsto &  \sigma_{2n+1}^c \nonumber
\end{eqnarray} 
which splits the natural map $D^1/D^4 \g \rightarrow D^1/D^2\g=\dg^1$. Consider
\begin{equation}\label{dgwedgedgtod2}
 \dg^1 \wedge \dg^1 \overset{\sigma^c\wedge \sigma^c}{\To} D^1/D^4 \g \wedge D^1/D^4 \g \overset{\{\ , \ \}}{\To} D^2 /D^5 \g\ .
 \end{equation}
The  subspace $K$ maps into $D^3/D^5 \g$, since its image in $D^2/D^3 = \dg^2 $ is zero. Since $K$ has even weights,  the  depth-parity theorem $\ref{corDepthParity}$ implies that
$D^3/D^4\g= \dg^3$ vanishes in even weights, and the restriction of $(\ref{dgwedgedgtod2})$ to $K$ gives  a linear map
\begin{equation} 
\label{cdef} \cc: K \To D^4/ D^5 \g = \dg^4\ .
\end{equation} 
The letter $\cc$ was chosen to stand for `cuspidal', for the following reason.
Its weight-graded components by $\cc_{2n}$ can be viewed,  via $(\ref{isomKtoS})$, 
as linear maps
$$\cc_{2n}: H_{\cusp}^1(\Gamma, V_{2n})^+ \To \dg_{2n}^4\  .$$

\begin{thm}  \label{thmccformula} Let  $P(x_1,x_2) = \sum_{i,j} \lambda_{i,j} x_1^{2i}x_2^{2j}$ be in $K$,  where $\lambda_{i,j} =-  \lambda_{j,i}$.  It gives rise to a relation
of the form
$$\sum_{i<j}  \lambda_{i,j} [ \overline{\sigma}_{2i+1}, \overline{\sigma}_{2j+1}] = 0  \quad \hbox{ in } \quad \dg^2\ .$$
Then  the image of the element $\cc(P) \in \dg^4$ in  $ \Q[x_1,x_2,x_3,x_4]$ is 
\begin{eqnarray}   \label{cPformula}
\rho^{(4)}( \cc(P)) & = &    \sum_{i,a,b}  \lambda_{i, a+b}   {B_{2a} B_{2b} \over B_{2a+2b}} \binom{2a+2b}{2a} {1 \over 24b} \{x_1^{2i}, \{x_1^{2a}, \{x_1^{2b}, x_1^{-2}\}\}\}   \\ 
 &  -  &   3  \sum_{ i } \lambda_{i,2} \{x_1^{2i},z_3 \} \nonumber 
\end{eqnarray}
where $z_3$ was defined in remark $\ref{remz_3}$.
\end{thm} 
\begin{proof} The element $\cc(P)$ is by definition
$$\cc(P) = \sum_{i<j}  \lambda_{i,j} \,\{ \sigma^c_{2i+1},  \sigma^c_{2j+1}\}   \pmod{D^5}$$
Now substitute the expressions $(\ref{sigmacexpl})$  for  $\sigma_{2j+1}^c$ in terms of the polar elements $\xi_{2a+1}$ (work in $\Q(x_1,\ldots, x_4)$).
By corollary $\ref{corquad}$, the $\xi_{2a+1}$ satisfy the relations
$$ \Big(  \sum_{i<j}  \lambda_{i,j} \,\{ \xi_{2i+1},  \xi_{2j+1}\} \Big)^{(r)}  =0  \ $$ 
for $1\leq r \leq 3$, where  a superscript $(r)$ denotes the depth $r$ component.
The theorem follows from formula $(\ref{sigmacexpl})$, together with the  definition of the element $z_3 =  -{1\over 3} \xi_3^{(3)}$.
 \end{proof}

If one believes the Broadhurst-Kreimer conjecture, one is led to the following\begin{conj}  \label{conj1}  (Broadhurst-Kreimer: compare with \cite{BrDepth}, \S9)
\begin{eqnarray} \label{homologyconj}
H_1(\dg; \Q) & \cong & \bigoplus_{n \geq 1} \overline{\sigma}_{2n+1} \Q \oplus    \cc(K)  \\
H_2(\dg; \Q) & \cong &K \nonumber \\
H_i(\dg; \Q) & =  & 0  \quad \hbox{for all} \quad  i \geq 3\ .\nonumber 
\end{eqnarray}
\end{conj}
 Thus $\dg$  admits the following  conjectural presentation. It has   generators  the $\overline{\sigma}_{2n+1}$ in depth 1 for $n \geq 1$ together with cuspidal elements $\cc(K)$ in depth 4. The only relations are   the quadratic  relations   of \S\ref{sectQuadrel}.

\begin{rem}
As noted  in \cite{EL}, $H_3(\dg;\Q)=0$ implies that $H_i(\dg;\Q)=0$ for all $i\geq 3$. 
In fact, for any pro-nilpotent Lie algebra  $\mathfrak{g}$ over a field $k$ of characteristic zero, $H_i(\mathfrak{g},k)=0$ implies that
$H_n(\mathfrak{g},k)=0$ for all $n\geq i$. To see this, note that  since $\mathfrak{g} $ is a projective limit of  finite-dimensional nilpotent Lie algebras, and (co)homology commutes with limits,  we can assume $\mathfrak{g}$  nilpotent and  $H^i(\mathfrak{g},k)=0$.  Every finite-dimensional $\mathfrak{g}$-module $M$  has an increasing filtration by submodules   $M_m \subset M$ such that the associated graded is a trivial module.
By the long exact cohomology sequence and induction on $m$,   $H^i(\mathfrak{g};M)=0$ for all such $M$.  Now interpret $H^n(\mathfrak{g};M)$
as the Ext group  $\mathrm{Ext}^n(k, M) $ in the category of $U \mathfrak{g}$-modules, and use the well-known fact that if $\mathrm{Ext}^i(k,M)$ vanishes for 
all $M$ then it also vanishes for all $n\geq i$.  
\end{rem}

The conjecture given in \cite{BrDepth}  involved certain exceptional generators denoted $\e_f$, for $f \in P$,  in the depth 4 component of the  larger Lie algebra $\gr^4_{D} \dmr$ of double shuffle equations. It is not known if they  are in  the image of  $\dg^4$.   Thus the formulation
 $(\ref{homologyconj})$ eliminates part of the conjecture given in \cite{BrDepth}.

\subsection{Remarks on the role of $z_3$}
The element $z_3$ is the first of a sequence $z_{2n+1}$ of derivations in $\Der^{\Theta} \, \LL(a,b)$ which 
are $\mathfrak{sl}_2$-invariant and  well-defined modulo $(\ue)^{\mathfrak{sl}_2}$.  It follows from theorem 10.1 in  \cite{MMV} that their action on the derivations $\varepsilon_{2k+2}^{\vee}$ are
known explicitly modulo  Lie brackets involving at least three $\varepsilon_{2n+2}^{\vee}$, with $n\geq 0$. 
It is possible that this computation can  be extended to the next order, which in particular would give  a  formula for $\{z_3, x_1^{2n}\}$ 
for all $n\geq 1$.
\begin{rem}  In \cite{BrDepth} we defined an injective linear map 
$$\overline{\e}: P \To \ls^4$$
from the space $P$ of even period polynomials to the space of solutions $\ls^4$ to the  linearised double shuffle equations in depth 4.  It only depends on the 
functional equations satisfied by elements of $P$.  It is natural to extend this linear map to the polynomials $x_1^{2n} -x_2^{2n} \in V_{2n}$,
which correspond to  coboundaries  under the morphism  $(\ref{Z1topolys})$. 
 Since they  satisfy the same functional equations as elements of $P$, they  define elements of $\mathfrak{pls}^4$ which have poles.  
One easily  verifies from the definitions that:
\begin{equation} 
\overline{\e}(x_1^{2n}-x_2^{2n}) + \{z_3, x_1^{2n-2}\}=0 \ . 
\end{equation}
This gives a different interpretation of the role of $z_3$ in formula $(\ref{cPformula})$.  
\end{rem}

\subsection{Elliptic interpretation of the Broadhurst-Kreimer conjecture} \label{sectQBK}
We can transpose the previous conjecture into the Lie algebra $\Der^{\Theta}\,  \LL(a,b)$ as follows. Recall that the map
$i_1 : \g \rightarrow \Der^{\Theta} \,  \LL(a,b)$ $(\ref{i1def})$  is injective.  Since  $B$ cuts out the depth filtration on the image $i_1(\g)$  (corollary \ref{corBcutout})  we obtain an injective  morphism
$$i_1: \dg \rightarrow \gr_B\, \Der^{\Theta} \, \LL(a,b)\ . $$
We wish to describe the conjectural generators in $B$-degrees $1$ and $4$. 
For simplicity, we shall use the heretical normalisations $\underline{\varepsilon}_{2n}$ to simplify the statement. This has the side-effect  of rescaling the period polynomial relations.

 More precisely,
consider linear map 
\begin{eqnarray}
 \Q[x_1^2,x^2_2]  &\To & \Q[x_1^2,x_2^2] \nonumber  \\
 x_i^{2n} & \mapsto &  {(2n)! \over B_{2n}} x_i^{2n} \quad \hbox{ for } n\geq 1\nonumber  
 \end{eqnarray}
and let $\underline{K}$ denote the image of $K$.  Lemma \ref{lemquadrel}  and corollary \ref{corquad} imply that 
$P = \sum_{i,j} \lambda_{i,j} x_1^{2i}x_2^{2j} \in \underline{K} $
where $\lambda_{i,j} + \lambda_{j,i}=0$,  if and only if 
$$\sum_{i<j} \lambda_{i,j} \{\underline{\varepsilon}^{\vee}_{2i+2} , \underline{\varepsilon}^{\vee}_{2j+2} \} =0\ . $$  
Define, for all $P \in \underline{K}$, elements
$$\cc(P) = \sum_{i<j} \lambda_{i,j} \{\underline{\sigma}^c_{2i+1},  \underline{\sigma}^c_{2j+1}\}  \in  \gr^4_B \, \Der^{\Theta}\, \LL(a,b)\ ,$$
 and let $z_3 \in \gr^3_B \Der^{\Theta} \, \LL(a,b)$ denote  the derivation such that $\ell'(z_3)$ is the element $(\ref{remz_3})$. 
\begin{thm} For any $P \in \underline{K}$, 
$$\cc(P) =\sum_j \lambda_{2,j} \big[ z_3, \underline{\varepsilon}^{\vee}_{2j+2}
\big]  +  \sum_{i,a,b}  \lambda_{i, a+b}  {1 \over 2b} \big[ \underline{\varepsilon}^{\vee}_{2i+2}, \big[ \underline{\varepsilon}^{\vee}_{2a+2}, \big[ \underline{\varepsilon}^{\vee}_{2b+2} , 
\underline{\varepsilon}^{\vee}_0\big] \big]\big] \  .$$
\end{thm} 

The Broadhurst-Kreimer conjecture suggests the following:
\begin{conj} (Elliptic (geometric) Broadhurst-Kreimer conjecture)
\begin{eqnarray} \label{homologyconj}
H_1(\dg, \Q) & \cong &  \bigoplus_{n\geq 1}  \underline{\varepsilon}^{\vee}_{2n+2} \Q  \oplus \cc(\underline{K})  \\
H_2(\dg, \Q) & \cong &  \underline{K} \nonumber \\
H_i(\dg, \Q) & =  & 0  \quad \hbox{for all} \quad  i \geq 3\ .\nonumber 
\end{eqnarray}
\end{conj}
Note that $\underline{K}$ in the above is interpreted as the space of quadratic relations between the $ \underline{\varepsilon}^{\vee}_{2n+2}$, for $n \geq 1$.
This conjecture is equivalent to conjecture $(\ref{conj1})$ by \S\ref{sectDrinGeom}.

\subsection{Some related problems}
\begin{enumerate}
\item Show that the map $\cc(P): K \rightarrow (\dg^4)^{\mathrm{ab}}$ is injective. 
\item Relate  the elements $\cc(P)$ to the exceptional elements  $\e_f$ defined in \cite{BrDepth}.
\item Construct a basis of the space of motivic periods of $\MT(\Z)$ of motivic depth $2$ and even weight  out of motivic multiple zeta values of depth $\leq 4$, using the formula for the $\sigma^c_{2n+1}$ modulo depth $4$.
\end{enumerate}

\section{Some  motivation from the relative completion of $\SL_2(\Z)$} \label{sectfinal} 
I shall very briefly sketch how I arrived at  formula $(\ref{introexplicitforsigma})$ and $(\ref{introexplicitsigmaasepsilons})$ by considering double integrals of Eisenstein series.
  This explains why the coefficients for the explicit formula for the $\sigma_{2n+1}^c$ involve the odd period polynomials of Eisenstein series.

\subsection{} 
Denote the Hecke-normalised  Eisenstein series of weight $2k\geq 4$   by 
\begin{equation}\nonumber
E_{2k} (q) = - {B_{2k} \over 4k} + \sum_{ n \geq 0} \sigma_{2k-1}(n) q^n \ ,
\end{equation}
where  $\sigma_k(n)$ denotes the divisor function.
For any modular form $f(\tau)$ of weight $2k\geq 4$ for $\SL_2(\Z)$ we shall write (see \cite{MMV} for further details):
\begin{equation} \label{underlinefdefinition}
\underline{f}(\tau) = (2\pi i)^{2k-1} f(\tau) (X- \tau Y )^{2k-2} d \tau\ 
\end{equation} 
where $q = \exp(2 i \pi \tau)$. 
It is to be viewed as a global section of $V_{2k-2} \otimes \Omega^1_{\HH}$ over the upper-half plane $\HH$. In \cite{MMV}, \S5, we defined
regularised iterated integrals of Eisenstein series between cusps. Consider the double integrals:
\begin{equation} \label{doubleeis} 
\int_0^{\infty}  \underline{E}_{2m+2}(\tau)   \underline{E}_{2n+2}(\tau) \quad   \in \quad  V_{2m} \otimes V_{2n} \otimes \C
\end{equation}
along the geodesic path from $0$ to $\infty$ (suitably interpreted as the path $S$ from $\tone_{\infty}$, the unit tangential base point at the cusp to itself).
For each $k\geq 0$, there is a canonical morphism   of $\SL_2$-representations  (\cite{MMV}, \S2.4)
$$\partial^{k} : V_{2m} \otimes V_{2n} \To V_{2m+2n-2k}\ .$$
In this way the imaginary part of the  image of  $(\ref{doubleeis})$ under $\partial^1$ defines a homogeneous polynomial in $\R[X,Y]$ of degree $2m+2n-2$ whose coefficients can be described explicitly.  The method described in \cite{MMV}, \S11, computes this polynomial as the Petersen inner product of 
two (real analytic) modular forms. The part we are interested, by the unfolding method,  corresponds to the convolution   of two Eisenstein series, and yields a certain multiple of an odd zeta value. One knows that  the ratios of these coefficients  are the odd period polynomials of Eisenstein series.

First of all, we give the precise technical statement about the periods of double Eisenstein series, and then explain how this relates to the image of 
$\g$ in $\Der^{\Theta} \, \LL(a,b)$.

\subsection{Precise statement}  All the notation in this section is  borrowed from \cite{MMV}, \S11.
  Let $k\geq 1$ be odd, $a,b\geq 2$, and $w=2a+2b-2k-2$.  Set 
\begin{equation}
\widetilde{I}^k_{2a,2b}=  I^k_{2a,2b} + \delta^0 \partial^k ( \overline{v}_{2a} \cup \bbf_{2b} - \bbf_{2a} \cup \overline{v}_{2b})
\end{equation}
where $\delta^0$ is the boundary for  $0$-cochains,  and  for all $k\geq 2$, 
$$\overline{v}_{2k} = (2 \pi i)^{2k-1} v_{2k}$$
where $v_{2k}$ was defined in \cite{MMV}, $(10.7)$. 
Then I claim that $\widetilde{I}^k_{2a,2b}$ is cocycle  for $\SL_2(\Z)$  which vanishes on $T\in \SL_2(\Z)$. Then it satisfies 
\begin{equation}\label{IkpariedwithEis}
   \{ i\,   \widetilde{I}^k_{2a,2b}, e^0_w \}  =  6 (2 \pi i)^{-w-1} C^k_{a,b} \, \zeta(k+1) \zeta(2a-k-1)\zeta(2b-k-1)\, \zeta(k+w) \end{equation} 
 where $e^0_{w}$ the rational Eisenstein cocycle defined in \cite{MMV} \S7.3 and 
 $$C^k_{a,b} =   k! (2a-2)! (2b-2)!  (k+w-1)! . $$
The equation  $(\ref{IkpariedwithEis})$ can be written, using Euler's formula, in terms of  a product of three bernoulli numbers and a  single odd zeta value.
Note that we only require the case $k=1$ here. 
The proof is essentially the same as in \cite{MMV}  with  modifications to account for divergences.

\subsection{Putting the pieces together} We sketch the main ingredients. 
Let $G_B^{\rel}$ denote the  completion of $\SL_2(\Z)$ relative to its inclusion into $\SL_2(\Q)$. It acts on $\PP$ $(\ref{PPdefn})$ through a quotient, 
called the `Eisenstein quotient' $G_B^{\eis}$  by Hain and Matsumoto \cite{HaGPS, MEM}. Let us suppose for simplicity\footnote{This is not strictly required for the following argument} that its affine ring is  an ind-object of the category of mixed Tate motives over $\Z$. There is a natural map 
$$\SL_2(\Z) \rightarrow G_B^{\rel}(\Q) \rightarrow G_B^{\eis}(\Q)$$
and hence a morphism, for any $\gamma \in \SL_2(\Z)$,
\begin{eqnarray} 
\mathrm{Isom}_{\MT(\Z)} (\omega_{dR},\omega_B)  & \To&  G^{\eis}_{dR}  \label{isomtoGdr} \\
\phi & \mapsto & \phi(\gamma) \nonumber 
\end{eqnarray} 
where $G^{\eis}_{dR}$ denotes  the $\Q$-affine group scheme underlying  $G^{\eis}_B(\C)$ which coincides with the $\Q$-de Rham structure of\cite{HaDe}.
The map $(\ref{isomtoGdr})$ defines  a canonical  homomorphism   $\SL_2(\Z) \rightarrow G^{\eis}_{dR} (P^{\mm})$
where $P^{\mm} =  \Or (\mathrm{Isom}_{\MT(\Z)} (\omega_{dR}, \omega_B))$ is the ring of motivic periods of $\MT(\Z)$.
Now Hain and Matsumoto show \cite{MEM}  that  there is a splitting  
\begin{equation} \label{Geissplits} 
G^{\eis}_{dR} \cong U^{\eis}_{dR} \rtimes \SL_2
\end{equation} so  the composition   $\SL_2(\Z) \rightarrow G^{\eis}_{dR} (P^{\mm}) \rightarrow U^{\eis}_{dR} (P^{\mm})$  defines  a cocycle
$$C^{\mm} \in Z^1 (\SL_2(\Z),  U^{\eis}_{dR} (P^{\mm}))\ .$$
Its period $\per ( C^{\mm}) \in  Z^1 (\SL_2(\Z),  U^{\eis}_{dR} (\C))$  is the image of the `canonical cocycle' $\mathcal{C}$ defined in \cite{MMV} in the group 
$U^{\eis}(\C)$.  Its coefficients are given by certain linear combinations of  regularised iterated integrals of Eisenstein series.

There is a canonical isomorphism $\Q(n)_{dR} \overset{\sim}{\rightarrow} \Q(n)_B$  given by  the comparison isomorphism $\mathrm{comp}_{B,dR}$ scaled by a suitable power of $2 \pi i$.  Since  $\Or(\SL_2)$ is a direct sum of pure Tate motives, we have in particular an  isomorphism
$M_0 \Or(\SL_2)_{dR}  \overset{\sim}{\rightarrow} M_0 \Or(\SL_2)_B$ of $\Q$-vector spaces, where $M$ is the weight filtration.
 Now consider the map
$$\Or(G^{\eis}_{dR}) \rightarrow M_0 \Or(G^{\eis}_{dR}) = M_0 (\Or(\SL_2))_{dR} \cong M_0(\Or(\SL_2))_B \rightarrow \Or(G_B^{\eis})$$
where the first  map  is the natural map to $\Or(G^{\eis}_{dR}) /F^1$, and the  second follows from  $(\ref{Geissplits})$, 
Dually, this defines a linear map $ \Or(G_B^{\eis})^{\vee} \rightarrow \Or(G^{\eis}_{dR})^{\vee}$, and hence a map $\SL_2(\Z) \rightarrow \Or(G^{\eis}_{dR})^{\vee}$.
One verifies that the element $STS^{-1} \in \SL_2(\Z)$, being lower triangular, maps to the element $STS^{-1} \in \Or(\SL_2)_{dR}^{\vee} \subset \Or(G^{\eis}_{dR})^{\vee}$. By applying this construction to the coefficients of $C^{\mm}_{STS^{-1}}$, and sending 
 $U^{\eis}_{dR} \rightarrow \mathrm{Aut}(\PP)$  we obtain an element 
 $$ C^{dR}_{STS^{-1}} \in \mathrm{Aut}(\PP)  (\Or(G^{dR}))$$
in the $\Or(G^{dR})$-points of $\mathrm{Aut}(\PP)$.
On the other hand, the structure of mixed Tate motive over $\Z$ on $\PP$ defines a morphism
$$I_1: G^{dR} \To \mathrm{Aut}(\PP)\ .$$
Unravelling all the definitions, one verifies  that 
$C^{dR}_{STS^{-1}}= I_1^{STS^{-1}} I_1^{-1}.$  Taking a logarithm and applying the Baker-Campbell-Hausdorff formula gives 
a relation between $i_1$ and $C^{dR}_{STS^{-1}}$. The latter is determined from $C^{dR}_S$  and $C^{dR}_T$ by the cocycle relations. From this, and knowledge of $C_T$,  one can read off the first few coefficients in the  expansion of $i_1(\sigma_{2n+1})$, for $n\geq 2$ from
the coefficients of $\zeta^{dR}(2n+1)$ in $C^{dR}_{S}$. 
This is the mechanism by which information about the map $i_1$ can be computed from double Eisenstein integrals. A different approach  would be via a motivic version of the analytic arguments of \cite{En}, which give  a relation between the Drinfeld associator and the image of $C_{STS^{-1}}$ in $U^{\eis}_{dR}$.
Further details   will be  given in a joint work with Hain.

\bibliographystyle{plain}
\bibliography{main}

\end{document}